\documentclass[10pt]{article}

\usepackage{multirow}
\usepackage{amsfonts}
\usepackage{amsmath}
\usepackage{amssymb}
\usepackage{empheq}
\usepackage{latexsym}
\usepackage{amscd}
\usepackage{bbm}
\usepackage{url}
\usepackage{comment}
\usepackage{xcolor}
\usepackage{hyperref}
\usepackage[left=3.2cm,right=3.2cm, top=2.5cm,bottom=2.5cm,bindingoffset=0cm]{geometry}
\usepackage{graphicx}

\newcommand\addressT{\noindent\leavevmode
\medskip
\noindent
Kateryna Tatarko, \\
Department of Pure Mathematics,\\
University of Waterloo, \\
Waterloo, Ontario, Canada, N2L 3G1.\\
\texttt{\small
e-mail:  ktatarko@uwaterloo.ca}
}

\newcommand\addressW{\noindent\leavevmode
	\medskip
	\noindent
	Elisabeth M. Werner, \\
	Department of Mathematics,\\
	Case Western Reserve University,\\
	Cleveland, Ohio, USA, 44106. \\
	\texttt{\small
		e-mail:  elisabeth.werner@case.edu}
}

\date{}

\newcounter{theorem}[section]
\newtheorem{theorem}{Theorem}
\newtheorem{lemma}[theorem]{Lemma}
\newtheorem{cor}[theorem]{Corollary}
\newtheorem{defi}[theorem]{Definition}
\newtheorem{prop}[theorem]{Proposition}

\newcommand{\N}{\mathbb{N}}

\newcommand{\R}{\mathbb{R}}

\newcommand{\dist}{{\rm dist}}

\newcommand{\qw}{\mathcal{W}}
\newcommand{\qv}{\mathcal{V}}

\newcommand{\eps}{\varepsilon}

\def\vol{{\rm vol}}

\begin{document}

\title{$L_p$-Steiner quermassintegrals
}

\author{Kateryna Tatarko\thanks{Supported by Canadian Natural Sciences and Engineering Research Council (NSERC) Discovery Grant number 2022-02961}  and Elisabeth M. Werner\thanks{Partially supported by  NSF grant DMS-2103482}}

\date{}

\maketitle

\begin{abstract}
Inspired by an $L_p$ Steiner formula  for the $L_p$ affine surface area proved by Tatarko and Werner,  we define, in analogy to the 
classical Steiner formula, $L_p$-Steiner quermassintegrals. Special cases include the classical mixed volumes, the dual mixed
volumes, the $L_p$ affine surface areas and the mixed $L_p$ affine surface areas.
We investigate the  properties of the $L_p$-Steiner quermassintegrals in a special class of convex bodies.
In particular, we show that they are rotation and reflection invariant  valuations in this class of convex bodies with a certain degree of 
homogeneity.
Such valuations seem new and have not been observed before.
	\end{abstract}
\vskip 3mm

\section{Introduction and Results}

An extension of the classical Brunn Minkowski theory,  the {\it $L_p$ Brunn Minkowski theory}, was  initiated by Lutwak in the groundbreaking paper \cite{Lutwak96}. 
This theory  evolved rapidly over the last years and
due to a number of highly influential works, see, e.g., 
\cite{GaKoSch} - \cite{HLYZ},  
\cite{LYZ2000} -  \cite{PW1}, \cite{Schu}, \cite{WY2010} - \cite{Zhao},  it  is now  a focal part of modern convex geometry. 
Central objects  in  the $L_p$ Brunn Minkowski theory are the $L_p$ affine surface areas,
\begin{equation*} 
as_{p}(K)=\int_{\partial K}\frac{H_{n-1}(x)^{\frac{p}{n+p}}}
{\langle x, N(x)\rangle ^{\frac{n(p-1)}{n+p}}} d\mathcal{H}^{n-1}(x),
\end{equation*}
where $N(x)$ denotes the outer unit normal at $x \in \partial K$, the boundary of $K$, $H_{n-1}(x)$ is the Gauss curvature at $x$ and $\mathcal{H}^{n-1}$ is the usual surface area measure on $\partial K$.
A missing part in the $L_p$ Brunn Minkowski theory was an analog to the classical Steiner formula \cite{Gardner, SchneiderBook} which says that the volume of the outer parallel body $K+ t B^n_2$ of the convex body $K$ with the Euclidean unit ball $B^n_2$ is a polynomial in $t$,
\begin{equation}\label{Steiner polynomial}
\vol_n(K+tB^n_2) = \sum\limits_{i = 0}^n {n \choose i} W_i(K)t^i.
\end{equation}
The coefficients $W_i(K)$ are the classical {\em quermassintegrals}. If the boundary $\partial K$ of the body $K$ is sufficiently smooth, and hence the principal curvatures are well-defined, then $W_i(K)=\frac{1}{n} \int_{\partial K} H_{i - 1} d\mathcal{H}^{n - 1}$, where for $j = 1, \dots, n-1$ $H_j$ denotes the  $j$-th normalized elementary symmetric function of the principal curvatures, see (\ref{ESFPC}).

\vskip 3mm
\noindent
In \cite{TatarkoWerner} an  $L_p$ Steiner formula was proved for the $L_p$ affine surface area. Namely, 
if $K$ is $C^2_+$, then we have for all suitable $t$ and for all $p \in \mathbb{R}$, $p \neq -n$, that %$0 \leq t \leq \min\limits_{u\in S^{n - 1}} h_K(u)$, 
\begin{eqnarray} \label{Formel}
as_p(K + tB^n_2) = \sum\limits_{k = 0}^\infty \left[ \sum\limits_{m=0}^k  \binom{ \frac{n(1-p)}{n+p}}{{k-m}} \qw^p_{{m}, k}(K) \right] t^k.
\end{eqnarray}
The coefficients $\qw^p_{{m}, k}(K) $ 
are called $L_p$ Steiner coefficients and defined
for a (general) convex body $K$ in $\mathbb{R}^n$, for all $k, m \in \mathbb{N} \cup \{0\}$ as
\begin{eqnarray*} 
	&&\hskip -11mm   \mathcal{W}^p_{m,\, k}(K) = \nonumber \\
	&&\hskip -11mm \int\limits_{\partial K} \langle x, N(x) \rangle^{m - k + \frac{n(1-p)}{n+p}} H_{n-1}^{\frac p{n+p}}  \sum_{\substack{
			i_1, \dots, i_{n-1} \geq 0 \\
			i_1 + 2i_2 + \dots + (n-  1)i_{n-1}=m}}
	c(n, p, i(m))  \prod\limits_{j = 1}^{n - 1}  {n - 1\choose j}^{i_j} H_{j}^{i_j} \, d\mathcal{H}^{n-1}.
\end{eqnarray*}
The $c(n, p, i(m))$ are certain binomial coefficients, see (\ref{cnp})  for the details.
The  $\mathcal{W}^p_{m,\, k}$ involve integration on $\partial K$. We also define corresponding expressions 	$\mathcal{Z}^p_{m,\, k}$ in (\ref{calZp}) 
as integrals over the sphere. If $K$ is sufficiently smooth, we can pass from one expression to the other via the Gauss map. 
\vskip 2mm
\noindent
It was also shown in \cite{TatarkoWerner} that for $p=0$ the above formula  gives the classical Steiner formula \eqref{Steiner polynomial}.
In particular,  the sum appearing in (\ref{Formel}) is  then a finite sum. The sum is also finite for all $p=- \frac{n(l-1)}{l}$, $l \in \mathbb{N}$, \cite{TatarkoWerner}.
In general, the sum is infinite and thus proving (\ref{Formel}) requires a careful analysis of  convergence issues.  Moreover, (\ref{Formel}) also includes  
the Steiner formula for the Minkowski outer parallel body of the dual Brunn Minkowski theory of Lutwak \cite{Lu1} as a special case. 
\vskip 3mm
\noindent
In this paper we define in analogy to the classical Steiner formula,  for a general convex body $K$ in $\mathbb{R}^n$ the {\em $L_p$-Steiner quermassintegrals}, 
or {\em $L_p$-Steiner mixed volumes}  as
\begin{equation}\label{LpSteiner-Ein}
\mathcal{V }^p_{k}(K) =  \sum\limits_{m=0}^k  \binom{ \frac{n(1-p)}{n+p}}{{k-m}} \qw^p_{{m}, k}(K).
\end{equation}
\vskip 2mm
\noindent
If $p=0$, then the $L_p$-Steiner quermassintegrals (\ref{LpSteiner-Ein}) are  the above classical quermassintegrals $W_i(K)$  (see \cite{Gardner, SchneiderBook}) and if $p = \pm \infty$,
they are  Lutwak's dual quermassintegrals \cite{Lu1, Lu1988}. Moreover, the $L_p$ mixed affine surface areas of Lutwak \cite{Lu}, see also \cite{WY2010}, are special cases of the $L_p$-Steiner quermassintegrals.
\vskip 2mm
\noindent
The $L_p$ Steiner coefficients  $\mathcal{W}^p_{m,\, k}$ and the $L_p$-Steiner quermassintegrals $\mathcal{V }^p_{k}$ are complicated expressions, and one needs to clarify when they are well-defined.  In Section \ref{sec_LpSteiner} we introduce a class of convex bodies, the class $\mathcal{C}^n$, for which this is the case.
\vskip 3mm
\noindent
We then investigate   in detail the properties of the $L_p$-Steiner quermassintegrals. Remarkably, it turns out  that they  are new valuations, which seem not to have been observed before. Valuations on convex sets can be considered as a generalization of the notion of a measure and thus are important quantities in the study of convex sets. Examples of valuations which are not measures in the usual sense are the quermassintegrals and the $L_p$ affine surface areas.  
Due to their importance, much work has been devoted to the study of valuations.
Often additional properties like  continuity  with respect to the Hausdorff metric and invariance under e.g., rigid motions are prescribed. 
Such valuations have been completely classified in a remarkable theorem by  Hadwiger  as linear combinations of the quermassintegrals in \cite{Hadwiger}. For a simpler proof, see \cite{Klain}. Subsequently, one direction of study in the theory of valuations concentrated on characterization theorems when one 
successively relaxes  the requirements. For instance, upper semi continuous, $GL(n)$-invariant valuations  have been characterized by Ludwig and Reitzner \cite{LudwigReitzner10}, involving,  in particular,  affine surface area.
\vskip 3mm
\noindent
The following theorem combines some of our main results, namely (parts of)  Theorems \ref{VProperties}  and \ref{WpmkProperties} below.  The class $\mathcal{C}^n$ is defined in Section~\ref{sec_LpSteiner}.
\vskip 2mm
\noindent
{\bf Theorem}
\par
\noindent
{\em Let $K$ be a convex body in $\mathcal{C}^n$ and let $p \in\mathbb{R}$ be such that $ p\geq 0$ or $p < -n$. . Then we have for all $ k \in \N$
	\begin{itemize}
		\item[(i)]  $\mathcal{V}^p_k $ is homogeneous of degree $n \frac{n - p}{n+p} - k$.
		\item[(ii)]  $\mathcal{V}^p_k $ is invariant  under rotations and reflections.
		\item[(iii)] $\mathcal{V}^p_k $ are valuations on the set~$\mathcal{C}^n$.
\end{itemize}	}	

\vskip 3mm
\noindent
In the next two corollaries we list a few consequences of our main theorem.
When $p=1$,  the expressions $\mathcal{V}^p_k$ simplify and we denote  $\mathcal{V}^1_k = \mathcal{W}^1_{k,\, k} = \mathcal{W}_{k,\, k}$. Additionally to the rotation and reflection invariance, one has translation invariance for  $\mathcal{W}_{k,\, k}$.
\vskip 2mm
\noindent
{\bf Corollary}
\par
\noindent
{\em For all $ k \in \N$, 
	$\qv^1_k = \qw_{k, \, k}$  are rigid motion invariant,  $n \frac{n - 1}{n+1} - k$ homogeneous valuations.
}	
\vskip 2mm
\noindent
For $m=0$, the $\qw^p_{{m}, k}$ are in addition upper or lower semi continuous (depending on the range of $p$). As these coincide with the mixed $L_p$ affine surface areas $as_{p+\frac{k}{n} (n+p), -k}$, this shows continuity properties of the latter, which, as far as we know, had not been noted before.
\vskip 2mm
\noindent
{\bf Corollary}
\par
\noindent
{\em For $p \geq 0$, $\mathcal{W}^p_{0,\, k}= as_{p+\frac{k}{n} (n+p), -k}$ are upper semi continuous $n \frac{n-p}{n+p} -k$ homogeneous valuations that are invariant under rotations and reflections. 
}
\vskip 2mm
\noindent
In view of these corollaries, we remark that a characterization of rigid motion invariant upper semi continuous valuations  has been given so far only in the plane \cite{Ludwig2000}.
Because of the last corollary and  the fact that the $L_p$ affine surface areas are semi continuous, 
it is natural to ask about 
continuity properties of the $\mathcal{V}^p_k$ and $\qw^p_{{m}, k}$. We  show in Section~\ref{continuity} that in general we cannot expect any semi continuity properties.

\vskip 2mm
\noindent
A byproduct of our analysis is following combinatorial formula. It may be known, but we did not find a reference for it.
\vskip 2mm
\noindent
{\bf Corollary}
\par
\noindent
{\em Let $p \in \mathbb{R}$, $p \neq -n$ and $k \in \mathbb{N}\cup \{0\}$. Then
	\begin{equation*}
	{n \frac{n-p}{n+p}\choose k} = \sum\limits_{m = 0}^{k} \binom{\frac{n(1-p)}{n+p}}{k-m} \sum_{\substack{
			i_1, \dots, i_{n-1} \geq 0 \\
			i_1 + 2i_2 + \dots + (n-  1)i_{n-1}=m}}
	c(n, p, i(m))  \prod\limits_{j = 1}^{n - 1}  {n - 1\choose j}^{i_j}
	\end{equation*}
where $c(n, p, i(m))$ are certain binomial coefficients, see (\ref{cnp})  for the details.
}

\section{Background}

\subsection{Background from differential geometry} \label {BackDG}

For more information and the details in this section we refer to e.g., \cite{Gardner,  SchneiderBook}.
\vskip 2mm

For a point $x$ on the boundary $\partial K$ of $K$ we denote by $N(x)$ an outward unit normal vector of $K$ at $x.$ Occasionally, we use $N_K(x)$ to emphasize that it is the normal vector of a body $K$ at $x \in \partial K$.  The map $N: \partial K \to S^{n - 1}$ is called the spherical image map or Gauss map of $K$ and it is of class $C^1.$ Its differential is called the Weingarten map. The eigenvalues of the Weingarten map are the principal curvatures $k_i(x)$ of $K$ at $x.$

The $j$-th normalized elementary symmetric functions of the principal curvatures are denoted by $H_j$. They are defined as follows
\begin{equation}\label{ESFPC}
H_j = {n - 1 \choose j}^{-1} \sum_{1\leq i_1 < \dots < i_j \leq n - 1} k_{i_1} \cdots k_{i_j}
\end{equation}
for $j = 1, \dots, n-1$ and $H_0 = 1.$ Note that 
$$H_1 = \frac{1}{n-1}  \sum_{1\leq i \leq n - 1} k_{i} $$ is the mean curvature, that is the average of principal curvatures, and 
$$H_{n - 1}= \prod_{i=1}^{n-1} k_i$$ 
is the Gauss curvature. 

\par
We say that $K$ is of class $C^2_+$ if $K$ is of class $C^2$ and the Gauss map $N$ is a diffeomorphism. This means in particular that  $N$ has a smooth inverse. This is stronger than just $C^2,$ and  is equivalent to the assumption that all principal curvatures are strictly positive, or that the Gauss curvature $H_{n-1} > 0.$
It also means that the differential of $N$, i.e., the Weingarten map, is of  maximal rank everywhere. 
\par
Let $K$ be of class $C^2_+$. For $u \in \R^n \setminus\{0\}$, let $\xi_K(u)$ be the unique point on the boundary of $K$ at which $u$ is an outward  normal vector. 
The map $\xi_K$ is defined on $\R^n \setminus \{0\}$. Its restriction to the sphere $S^{n - 1}$ is called the reverse spherical image map, or reverse Gauss map, $N_K^{-1}: S^{n - 1} \to \partial K$.  The differential of $N_K^{-1}$ is called the reverse Weingarten map. The eigenvalues of  the reverse Weingarten map are called the principal radii of curvature $r_1, \dots, r_{n - 1}$ of $K$ at $u\in S^{n - 1}.$ 
\par
The $j$-th normalized elementary symmetric functions of the principal radii of curvature are denoted by $s_j$. In particular, $s_0=1$,   and for $1 \leq j \leq n-1$ they are defined as
\begin{equation}\label{ESFPR}
s_j = {n - 1 \choose j}^{-1} \sum_{1\leq i_1 < \dots < i_j \leq n - 1} r_{i_1} \cdots r_{i_j}.
\end{equation}
Note that the principal curvatures are functions on the boundary of $K$ and the principal radii of curvature are functions on the sphere. 
\par
Now we describe  the connection between $H_j$ and $s_j.$ For a body $K$ of class $C^2_+$, the principal radii of curvature are reciprocals of the principal curvatures, that is
$$
r_i(u) = \frac1{k_i(N_K^{-1}(u))}.
$$
This implies that for $x \in \partial K$ with $N(x)=u$, 
$$
s_j = {n - 1 \choose j}^{-1} \sum_{1\leq i_1 < \dots < i_j \leq n - 1} \frac1{k_{i_1}(N_K^{-1}(u)) \cdots k_{i_j}(N_K^{-1}(u))} = \frac{H_{n - 1 -j}}{H_{n - 1}}\Big(N_K^{-1}(u)\Big)
$$
and
$$
H_j  =  \frac{s_{n - 1 -j}}{s_{n - 1}}\Big(N(x)\Big), 
$$
for $j = 1, \dots, n-1.$
\par
The mixed volumes $W_i(K)$ of the classical Steiner formula (\ref{Steiner polynomial}) can be expressed via the elementary symmetric functions of the principal curvatures.  
By definition, $W_0(K) = \vol_n(K)$ and
\begin{equation}\label{QMI}
W_i(K) = \frac1n \int\limits_{\partial K} H_{i - 1} d\mathcal{H}^{n - 1}, 
\end{equation} 
for $i = 1, \dots, n.$
\vskip 2mm
\noindent
The dual mixed volume  of the convex bodies $K$ and $L$ that contain $0$ in their interiors, introduced by Lutwak \cite{Lu1}, is defined for all real  $i$  by
$$
\widetilde{V}_i (K, L) = \frac1n \  \int\limits_{S^{n - 1}} \rho_K(u)^{n - i} \rho_L(u)^i d\sigma(u), 
$$
where $\rho_K(u) = \max\{\lambda \geq 0 |\, \lambda u \in K\}$ is the radial function of $K.$ In particular, if $L = B^n_2$, then
\begin{equation}\label{dualQMI}
\widetilde{W}_i(K) = \widetilde{V}_i (K, B^n_2) = \frac{1}{n} \int\limits_{S^{n - 1}} \rho_K(u)^{n - i} d\sigma(u)
\end{equation}
are called {\it dual quermassintegrals} of order $i.$ The corresponding Steiner  formula in the dual Brunn Minkowski theory proved in \cite{Lu1} is
\begin{equation*}
\vol_n(K \, \widetilde{+} \, tB^n_2) = \sum\limits_{i = 0}^n {n \choose i} \widetilde{W}_i(K) \  t^i .
\end{equation*}
\par

%%%%%%%%%%%%%%%%%%%%%%%%%%%%%%%%%%%%%%%%

\subsection{Background from affine geometry} \label{BackAff}

We denote by $\mathcal{K}^n_o$ the set of all convex bodies in $\mathbb{R}^n$ containing the origin $o$. From now on, we will always assume that $K \in \mathcal{K}^n_o$. For real  $p \neq -n$, we define  the
$L_p$ affine surface area $as_{p}(K)$ of $K$ as in \cite{Lu1} ($p
>1$) and \cite{SW2004} ($p <1,\ p \ne -n$) by
\begin{equation} \label{def:paffine}
as_{p}(K)=\int\limits_{\partial K}\frac{H_{n-1}(x)^{\frac{p}{n+p}}}
{\langle x,N(x)\rangle ^{\frac{n(p-1)}{n+p}}} \, d\mathcal{H}^{n-1}(x)
\end{equation}
and
\begin{equation}\label{def:infty}
as_{\pm\infty}(K)=\int\limits_{\partial K}\frac{H_{n-1} (x)}{\langle
	x,N(x)\rangle ^{n}} d\mathcal{H}^{n-1}(x)
\end{equation}
provided the above integrals exist. Note that these expressions are not always finite. For example, if $K$ is a polytope and $-n < p < 0$, then $as_p(K) = \infty$. 

In particular, for $p=0$, 
\begin{equation}\label{0-asa}
as_{0}(K)=\int\limits_{\partial K} \langle x,N(x)\rangle
\,  d\mathcal{H}^{n-1}(x) = n\, \vol_n(K).
\end{equation}
The case $p=1$, 
$$
as_{1}(K)=\int\limits_{\partial K}H_{n-1}(x)^{\frac{1}{n+1}}  d\mathcal{H}^{n-1}(x)
$$
is the classical affine surface area  which is independent
of the position of $K$ in space.  For dimensions 2 and 3 and sufficiently smooth convex bodies, its definition goes 
back to Blaschke \cite{Blaschke}.
\par
\noindent
If the boundary of $K$ is sufficiently smooth, then (\ref{def:paffine})
and (\ref{def:infty}) can be written as integrals over the boundary
$\partial B^n_2=S^{n-1}$ of the Euclidean unit ball $B^n_2$  in $\mathbb R^n$,
\begin{equation} \label{pasa}
as_{p}(K)=\int\limits_{S^{n-1}}\frac{f_{K}(u)^{\frac{n}{n+p}}}
{h_K(u)^{\frac{n(p-1)}{n+p}}}
d\mathcal{H}^{n-1}(u),
\end{equation}
where $f_{K}(u)$ is the curvature function, i.e. the reciprocal of
the Gaussian curvature $H_{n-1}(x)$ at this point $x \in
\partial K$ that has $u$ as outer normal, and $h_K(u) = \sup\{ \langle x, u \rangle: x \in  K \}$, $u \in \mathbb{R}^n \setminus\{0\}$,  is the support function of $K$. In particular, for
$p=\pm \infty$,
\begin{equation}\label{inf-aff}
as_{\pm\infty}(K)
=\int\limits_{S^{n-1}}\frac{1}{h_K(u)^{n}}
d\mathcal{H}^{n-1}(x)
=n \, \vol_n(K^{\circ}), 
\end{equation}
where $K^\circ=\{y\in \mathbb{R}^n: \langle x, y\rangle\leq 1 \  \text{for all } x\in K\}$ is the polar body of $K$. 
\vskip 2mm
\noindent
For $p=-n$, the $L_{-n}$ affine surface area was introduced in \cite{MW2} as
\begin{equation}\label{L-n}
as_{-n}(K) = \max _{u \in S^{n-1}} f_K(u)^\frac{1}{2} h_K(u)^\frac{n+1}{2}.
\end{equation}
Hug \cite{Hug} proved that \eqref{def:paffine} and \eqref{pasa} coincide for $p>0$.

%%%%%%%%%%%%%%%%%%%%%%%%%%%%%%%%%%%%%%%%%%%%%%%
%%%%%%%%%%%%%%%%%%%%%%%%%%%%%%%%%%%%%%%%%%%%%%%%

\section{$L_p$ Steiner quermassintegrals}\label{sec_LpSteiner}

\subsection{Definitions}

We will need the  generalized binomial coefficients. For $\alpha \in \mathbb{R}$ and $k \in \mathbb{N}$, they are defined as
\begin{equation} \label{gcoef}
{\alpha \choose k} =
\left\{
\begin{aligned}
&1 \quad &\text{ if }\  k=0,\\
&0 \quad &\text{ if }\,  k<0 \,  &\text{ or }\, \alpha =0,\\
&\frac{\alpha(\alpha - 1)\cdots(\alpha - k + 1)}{k!} \quad  &\text{ if } k > 0.
\end{aligned}\right.
\end{equation}
\par
\noindent
\vskip 3mm
\noindent 
It was shown in \cite{TatarkoWerner} that the following {\em $L_p$ Steiner formulas} hold. They are a generalization of the classical Steiner formula, which corresponds to the case $p=0$ and Lutwak's Steiner formula of the dual Brunn Minkowski theory, which corresponds to the case $p=\pm \infty$: 
\par
\noindent
\begin{theorem} {\cite{TatarkoWerner}}\label{TatarkoWerner}
Let  $k,m \in \mathbb{N}\cup \{0\}$, and $p \in \mathbb{R}$, $p \neq -n$. 
 If $K \in \mathcal{K}_o^n$ is $C^2_+$, then we have for all $0 \leq t \leq \min\limits_{u\in S^{n - 1}} h_K(u)$, 
	\begin{eqnarray}\label{Steiner as_p}
	as_p(K + tB^n_2) = \sum\limits_{k = 0}^\infty \left[ \sum\limits_{m=0}^k  \binom{ \frac{n(1-p)}{n+p}}{{k-m}} \qw^p_{{ m}, k}(K) \right] t^k 
	=\sum\limits_{k = 0}^\infty \left[ \sum\limits_{m = 0}^{k} \binom{\frac{n(1-p)}{n+p}}{k-m} \mathcal{Z }^p_{m,\, k}(K) \right]t^k. 
	\end{eqnarray}
\end{theorem}
A convex body $K$ has a unique outer normal almost everywhere, and by a theorem of Alexandroff~\cite{Alexandroff} and Busemann-Feller \cite{Buse-Feller} the generalized second partial derivatives exist almost everywhere.  Therefore, we extend the definition of the $L_p$ Steiner coefficients from smooth convex bodies \cite{TatarkoWerner} to general convex bodies $K$.
\begin{defi} [$L_p$-Steiner coefficients] Let $K \in \mathcal{K}_o^n$, $k,m \in \mathbb{N}\cup \{0\}$, and $p \in \mathbb{R}$, $p \neq -n$. Then we define
\begin{eqnarray} \label{calWp}
	&&\hskip -11mm  \, \,  \mathcal{W}^p_{m,\, k}(K) = \nonumber \\
	&&\hskip -11mm \int\limits_{\partial K} \langle x, N(x) \rangle^{m - k + \frac{n(1-p)}{n+p}} H_{n-1}^{\frac p{n+p}}  \sum_{\substack{
			i_1, \dots, i_{n-1} \geq 0 \\
			i_1 + 2i_2 + \dots + (n-  1)i_{n-1}=m}}
	 c(n, p, i(m))  \prod\limits_{j = 1}^{n - 1}  {n - 1\choose j}^{i_j} H_{j}^{i_j} \, d\mathcal{H}^{n-1}
	\end{eqnarray}
and 
 \begin{eqnarray} \label{calZp}
	&&\hskip -11mm  \, \, 
	\mathcal{Z}^p_{m,\, k}(K) = \nonumber \\
	&&\hskip -11mm \int\limits_{S^{n-1}} h_K^{m - k + \frac{n(1-p)}{n+p}} \,  s_{n-1}^{\frac n{n+p}} (u) \sum_{\substack{
			i_1, \dots, i_{n-1} \geq 0 \\
			i_1 + 2i_2 + \dots + (n-  1)i_{n-1}=m}}
	c(n, p, i(m))  \frac{ \prod\limits_{j = 1}^{n - 1}  {n - 1\choose j}^{i_j} s_{n-1-j}^{i_j}\left(u \right)}{s_{n-1}^{\sum\limits_{j} i_j}(u)}  \, d\mathcal{H}^{n-1}. 
	\end{eqnarray}
\end{defi}
\vskip 2mm
\noindent
There we have put
\begin{equation} \label{cnp}
c(n, p, i(m)) = {\frac {n}{n+p} \choose i_1 + \dots + i_{n - 1}} {i_1 + \dots + i_{n - 1}\choose i_1,\, i_2, \, \ldots, i_{n-1}},
\end{equation}
where the sequence $i(m) = \{i_j\}_{j = 1}^{n-1}$ is such that $i_1 + 2 i_2 +\dots + (n-1)i_{n - 1} =m$,
and 
\begin{equation*}
{q \choose i_1, i_2,\dots, i_l} = \frac{q!}{i_1!i_2!\cdots i_l!}
\end{equation*}
is the multinomial coefficient where $q = i_1 + \dots + i_l$. Note that 
\begin{equation*} \label{mcoef}
{q \choose i_1, i_2,\dots, i_l} = 0 \quad \text{ if }\  i_j<0 \text{ or } i_j>q.
\end{equation*}
When $l = 2,$ we get the binomial coefficients. 
\vskip 2mm

\vskip 4mm
\noindent	
{\bf Remark} (see \cite{TatarkoWerner})
\vskip 1mm
\noindent  If $p = 1$, the $L_p$ Steiner formula (\ref{Steiner as_p})  reduces to 
\begin{eqnarray}\label{Steiner as_1}
	as_1(K + tB^n_2) &=& \sum\limits_{k = 0}^\infty \mathcal{W}_{k,\, k}(K) \,t^k	 =  \sum\limits_{k = 0}^\infty \mathcal{Z}_{k,\, k}(K) \,t^k,
\end{eqnarray}	
and  the expressions (\ref{calWp}) and  (\ref{calZp}) simplify to 
\begin{eqnarray}\label{calW1}
	1. \, \,  \mathcal{W}_{k,\, k}(K) &=& \mathcal{W}^1_{k,\, k}(K) = \nonumber \\
	&=& \int\limits_{\partial K}  H_{n-1}^{\frac 1{n+1}}  \sum_{\substack{
			i_1, \dots, i_{n-1} \geq 0 \\
			i_1 + 2i_2 + \dots + (n-  1)i_{n-1}=k}}
	c(n, 1, i(k))  \prod\limits_{j = 1}^{n - 1}  {n - 1\choose j}^{i_j} H_{j}^{i_j}\left(x \right) \, d\mathcal{H}^{n-1}
\end{eqnarray}
and 
\begin{eqnarray}\label{calU1}
	2. \, \, 	\mathcal{Z}_{k,\, k} (K) &=& \mathcal{Z}^1_{k,\, k} (K) = \nonumber\\
	&=& \int\limits_{S^{n-1}} s_{n-1}^{\frac n{n+1} }  \sum_{\substack{
			i_1, \dots, i_{n-1} \geq 0 \\
			i_1 + 2i_2 + \dots + (n-  1)i_{n-1}= k}}
	c(n, 1, i(k))  \frac{ \prod\limits_{j = 1}^{n - 1}  {n - 1\choose j}^{i_j} s_{n-1-j}^{i_j}}{s_{n-1}^{\sum\limits_{j} i_j}(u)}  \, d\mathcal{H}^{n-1}.
\end{eqnarray}

\medskip

\par
\noindent The natural question is for which class of bodies and for which parameters the quantities \eqref{calWp}  and \eqref{calZp} are well-defined. We are investigating this next. To do so, we introduce the notions of inner and outer rolling balls.

We say that $K$ admits an {\it inner rolling ball} of radius $r > 0$ if for any $x \in \partial K$
$$
B^n_2(x - r N(x), r) \subseteq K,
$$
and $K$ admits an {\it outer rolling ball} of radius $R>0$ if for any $x \in \partial K$
$$
K \subseteq B^n_2(x - R N(x), R).
$$
We then define the following class of convex bodies: 
\begin{align}\label{class1}
	\mathcal{C}^n = \{K \in \mathcal{K}_o^n:\   K \ \textup{admits an inner and outer rolling ball}\}.
\end{align}
The class $\mathcal{C}^n$ proved useful in many contexts, e.g. in approximation of convex bodies by random polytopes \cite{SchuettWerner2003}.
Note that if $K\ \in \mathcal{C}^n$ then $K$ is $C^1$ and strictly convex. Moreover, if $K$ is $C^2_+$ then $K$ is in $\mathcal{C}^n$ but the converse does not hold in general.
We discuss the class $\mathcal{C}^n$ in detail in Section~\ref{first analysis}. In particular, we show that the $L_p$ Steiner coefficients (for $p \in \mathbb{R}$ such that $\frac{p}{n+p} \geq 0$) are finite for any $K \in \mathcal{C}^n$. As $K \in \mathcal{C}^n$ is $C^1$ we can pass from $\mathcal{W}^p_{m,k}$ to $\mathcal{Z}^p_{m,\, k}$ via the Gauss map $N: \partial K \to S^{n-1}$, and as $K$ is strictly convex, we can pass from 	$\mathcal{Z}^p_{m,\, k}$ to $\mathcal{W}^p_{m,k}$ via the inverse of the Gauss map. Thus, it is enough to consider $\mathcal{W}^p_{m,k}$.

\vskip 3mm
\noindent	
In analogy to the classical Steiner formula \eqref{Steiner polynomial},  Theorem~\ref{TatarkoWerner} leads us to define the  $L_p$-Steiner quermassintegrals.
\par
\noindent
\begin{defi} [$L_p$-Steiner quermassintegrals]
Let $K$ be a convex body in $\mathbb{R}^n$,   $p \in \R$, $p \ne -n$ and  $k \in \N \cup  \{0\} $. Then we define the  $L_p$-Steiner quermassintegrals
	\begin{equation} \label{p-SC1}
	\mathcal{V}_k^p(K) = \sum\limits_{m = 0}^{k} \binom{\frac{n(1-p)}{n+p}}{k-m} \mathcal{W}^p_{m,\, k}(K),
	\end{equation}
	and
		\begin{equation}  \label{p-SC2}
	\mathcal{U}_{k}^p(K) = \sum\limits_{m = 0}^{k} \binom{\frac{n(1-p)}{n+p}}{k-m} \mathcal{Z }^p_{m,\, k}(K).
	\end{equation} 
\end{defi}
\vskip 2mm
\noindent
We will mainly consider convex bodies that are in the class $\mathcal{C}^n$ and $p\in \mathbb{R}$ such that $\frac{p}{n+p} \geq 0$. Then we can pass from $\mathcal{U}^p_{k}$ to $\mathcal{V}^p_{k}$ via the Gauss map $N: \partial K \to S^{n-1}$ and as $K$ is strictly convex, we can pass from 	$\mathcal{V}^p_{k}$ to $\mathcal{U}^p_{k}$ via the inverse of the Gauss map. Thus, it suffices to consider $\mathcal{V}^p_{k}$.
\vskip 2mm
\noindent
If $p=1$, then
	\begin{equation*}
	\mathcal{V}_k^1(K) = \mathcal{W}_{k,\, k}(K)\hskip 10mm \text{and} \hskip 10mm \mathcal{U}_k^1(K) = \mathcal{Z}_{k,\, k}(K).
	\end{equation*}	
\vskip 2mm
\noindent
In general, the expressions $\mathcal{V}_k^p(K)$ may become infinite or undetermined, depending on the body $K$ and the parameter $p$. In Section \ref{first analysis}, we  will discuss this issue and we will show  that $\mathcal{V}_k^p(K)$ are finite for $K \in \mathcal{C}^n$ and $p\in \mathbb{R}$ such that $\frac{p}{n+p} \geq 0$. 
\vskip 4mm
\noindent	
{\bf Remark}
\vskip 1mm
\noindent
1. Note that the $L_p$-Steiner quermassintegrals as well as $L_p$ Steiner coefficients can be negative. This is not the case for the classical quermassintegrals.

\medskip

\noindent 2. Nevertheless, the $L_p$-Steiner quermassintegrals closely parallel the classical Steiner coefficients. This is further illustrated next. In the classical Steiner  formula (\ref{Steiner polynomial}), the first and the last coefficients are
$$W_0(K) = \vol_n(K) \,\, \, \text{ and} \,\, \, \, 
W_n(K) = \vol_n(B^n_2),
$$ 
respectively. In the $L_p$ Steiner formula (\ref{Steiner as_p}),  the first coefficient is $\mathcal{V}_0^p(K) = as_p(K)$. It was shown in \cite{TatarkoWerner} that  when  $p= \frac{-n(l-1)}{l}$, $ l \in \mathbb{N}$,  the sums (\ref{Steiner as_p}) are finite with the highest term $t^{n(2l-1)}$:
\begin{align*}
&as_{-\frac{n(l-1)}l}(K + tB^n_2) = \sum_{k=0}^{l(n-1)} \left[\sum_{m=0}^{k} \binom{l+n(l-1)}{k-m} \qw_{m, \, k}(K) \right]t^k \,+
\\
&\sum_{k=l(n-1)+1}^{l+n(l-1)} \left[\sum_{m=0}^{l(n-1)} \binom{l+n(l-1)}{k-m} \qw_{m, \, k} (K) \right]t^k 
+\sum_{k=l+n(l-1)+1}^{n(2l-1)} \left[\sum_{m=k-(l+n(l-1))}^{l(n-1)} \binom{l+n(l-1)}{k-m} \qw_{m, \, k} (K) \right]t^k.
\end{align*}
Then the last coefficient is 
\begin{align*}
\int\limits_{S^{n-1}}  s_{n-1}^{l} (u) \sum_{\substack{
		i_1, \dots, i_{n-1} \geq 0 \\
		i_1 + \dots + (n-  1)i_{n-1}=l(n-1)}}
{l \choose i_1 + \dots + i_{n - 1}} {i_1 + \dots + i_{n - 1}\choose i_1,\, \ldots, i_{n-1}}  \frac{ \prod\limits_{j = 1}^{n - 1}  {n - 1\choose j}^{i_j} s_{n-1-j}^{i_j}\left(u \right)}{s_{n-1}^{\sum\limits_{j} i_j}(u)}  \, d\mathcal{H}^{n-1}(u). 
\end{align*}
We only get a contribution in the above sum if $i_1 + \dots + i_{n - 1} \leq l$ and $i_1 + 2i_2 + \dots + (n-  1)i_{n-1}= l(n-1).$ Therefore,
$$
i_1 + 2i_2 + \dots + (n- 2)i_{n-2}= (l-i_{n-1})(n-1) \geq (n-1)(i_1 + \dots + i_{n - 2}),
$$
since $i_1 + \dots + i_{n - 1} \leq l.$ This is true only when $i_1 = i_2 = \dots = i_{n-2} = 0$. Thus, the only possible choice of indices is $ i_1 = i_2 = \dots = i_{n-2} = 0$ and $i_{n-1} = l.$ Then the coefficient of $t^{n(2l-1)}$, i.e. the last coefficient, is
\begin{align*}
\int\limits_{S^{n-1}} \, d\mathcal{H}^{n-1}(u) =\vol_{n-1}(\partial B^n_2) =  as_p(B^n_2),
\end{align*}
which shows the parallel behavior of the classical and $L_p$ Steiner coefficients. 

\vskip 4mm

\section {A first analysis of the $L_p$-Steiner quermassintegrals} \label{first analysis}

%%%%%%%%%%%%%%%%%%%%%%%%%%%%%%%
\subsection{The class $\mathcal{C}^n$}\label{gen_obs}

\begin{prop}
	Let $p \in\mathbb{R}$ be such that $ p\geq 0$ or $p < -n$. Let $k, m \in \mathbb{N} \cup \{0\}$ and let the sequence $i(m) = \{i_j\}_{j = 1}^{n-1}$ be such that $i_1 + 2 i_2 +\dots + (n-1)i_{n - 1} =m$. Let $\mathcal{C}^n$ be defined as in \eqref{class1}. Then we have for any $K \in \mathcal{C}^n$ 
	$$
	-\infty < \mathcal{W}^p_{m,\, k}(K) < \infty \qquad \textup{and} \qquad 	-\infty <\mathcal{V }^p_{k} (K) < \infty.
	$$ 
\end{prop}
\begin{proof}
	The   $L_p$ Steiner coefficients  $\mathcal{W}^p_{m,\, k}$ (\ref{calWp}), and thus  
	the $L_p$-Steiner quermassintegrals $\mathcal{V }^p_{k}$ (\ref{p-SC1}), are sums (up to  coefficients) of expressions of the form 
	\begin{equation}\label{Teil}
		\int\limits_{\partial K} \langle x, N(x) \rangle^{m - k + \frac{n(1-p)}{n+p}} H_{n-1}^{\frac p{n+p}}  \,  \prod\limits_{j = 1}^{n - 1}  {n - 1\choose j}^{i_j} H_{j}^{i_j}\, d\mathcal{H}^{n-1}.
	\end{equation}
As $K$ is a convex body in $\mathcal{K}_o^n$, there is $a>0$, $a \in \mathbb{R}$, such that 
$
B^n_2(0,a) \subset K \subset B^n_2\left(0, \frac{1}{a}\right)
$.
Thus 
\begin{equation}\label{supp}
a \leq \langle x, N(x) \rangle \leq \frac{1}{a}.
\end{equation} 
Since $K$ admits an inner rolling ball with radius $r$ and an outer rolling ball with radius $R$, we have for all $x\in \partial K$
$$
\frac1{R^{n-1}}\leq H_{n-1}(x) \leq \frac1{r^{n-1}}
$$
and each $k_i(x)$ is bounded from above and below by $\frac1r$ and $\frac1R$, respectively. This implies 
\begin{align*}
	 \gamma \left(\frac1R\right)^{m + (n-1) \frac{p}{n+p}} \vol_{n-1}(\partial K) \leq \int\limits_{\partial K}  H_{n-1}^{\frac p{n+p}}  \,  \prod\limits_{j = 1}^{n - 1}  {n - 1\choose j}^{i_j} H_{j}^{i_j}\, d\mathcal{H}^{n-1} \leq \gamma \left(\frac1r\right)^{m + (n-1) \frac{p}{n+p}} \vol_{n-1}(\partial K).
\end{align*} 
where $\gamma = \prod\limits_{j = 1}^{n - 1}  {n - 1\choose j}^{i_j}.$

Combining this with \eqref{supp}, gives that the $L_p$ Steiner coefficients $\mathcal{W}^p_{m,\, k}$ are finite. Thus the $\mathcal{V }^p_{k}$ are  also finite as a finite sum of $\mathcal{W}^p_{m,\, k}$. 
\end{proof}
\vskip 2mm
\par
\vskip 2mm
\noindent The next example shows that the assumptions on the parameter $p$ and the rolling ball assumptions are needed so that the $L_p$-Steiner quermassintegrals are finite.

\medskip

\noindent{\bf Example 4.2} \ Let $1 < r < \infty$ and $B_r^n = \{ x \in \mathbb{R}^2: |x_1|^r + \cdots + |x_n|^r \leq 1\}$. For $m = 0$
$$
\mathcal{W}^p_{0,\, k}(B_r^n) =  \int\limits_{\partial K} \langle x, N(x) \rangle^{ - k + \frac{n(1-p)}{n+p}} H_{n-1}^{\frac p{n+p}}  
\, d\mathcal{H}^{n-1} = \int\limits_{\partial K} \langle x, N(x) \rangle^{\frac{n(1-p)}{n+p}} H_{n-1}^{\frac p{n+p}}   \langle x, N(x) \rangle^{-k}
\, d\mathcal{H}^{n-1} .
$$
Using \eqref{supp}, we get 
$$
a^{k} as_p(B_r^n) \leq \mathcal{W}^p_{0,\, k}(B_r^n) \leq \frac{1}{a^k} as_p(B_r^n) .
$$
Now we observe the following:
\medskip

{\it
	(i) Let $2 < r < \infty$. Then $B^n_r$ admits inner and outer rolling balls. It was shown in \cite{SW2004} that for $-n < p < - \frac{n}{r-1}$
	$$
	as_p(B_r^n) = \infty.
	$$
	This shows that the assumption $\frac{p}{n+p} \geq 0$ is needed.
	
	(ii) Let $1 < r < 2$. Then   $B^n_r$ does not admit an inner rolling ball. It was shown in \cite{SW2004} that 
	$$
	as_p(B_r^n) = \infty
	$$
	for $-\frac{n}{r-1} \leq p < -n$. This case shows that the assumption on rolling balls is needed.
	 
}
\par
\noindent

\begin{prop}\label{int_union}
The class $\mathcal{C}^n$ is closed under intersections and unions assuming that the union is convex.
\end{prop}
\begin{proof}
	Let $K$ and $L$ be convex bodies in $\mathcal{C}^n$. First, we show that $K \cap L$ and $K \cup L$  admit an inner rolling ball. We start with $K \cup L$ and split its boundary into disjoint sets as follows
	$$
	\partial(K \cup L) = \{\partial K\cap\partial L\}\cup\{\partial K\cap L^{c}\}\cup\{K^{c}\cap \partial L\} .
	$$
	Then we consider the following cases:
	\begin{itemize}
		\item[1)] Let $x \in \partial K\cap\partial L.$ Since $x\in \partial K$ and $K$ is in $\mathcal{C}^n$, then $N_K(x)$ is unique. At the same time, since $x\in \partial L$ and $L$ is in $\mathcal{C}^n$, then $N_L(x)$ is unique. Hence, $N_K(x) = N_L(x) = N(x)$ and there are inner rolling balls of $K$ and $L$ at $x$ with radii $r_K$ and $r_L$, respectively. We take $r = \min\{r_K, r_L\}$. Then $B^n_2(x - r N(x), r) \subseteq K \cup L$.
		\item[2)] Let $x \in \partial K\cap L^{c}$. Then $B^n_2(x - r_K N_K(x), r_K) \subseteq K$. By taking $r = \min\{r_K, r_L\}$ and $N(x) = N_K(x)$, we get that $B^n_2(x - r N(x), r) \subseteq K$ which implies $B^n_2(x - r N(x), r) \subseteq K \cup L.$
		\item[3)] The proof that there is an inner rolling ball when $x \in K^{c}\cap \partial L$ is similar to the case 2).
	\end{itemize}	
	Next, we deal with $K \cap L$ and decompose its boundary as
	$$
	\partial(K \cap L) = \{\partial K\cap\partial L\}
	\cup\{\partial K\cap\operatorname{int}(L)\}\cup\{\operatorname{int}(K)\cap \partial L\}.
	$$
	Again, we consider the following cases:
	\begin{itemize}
		\item[1)] Let $x \in \partial K\cap\partial L.$ Then the proof is similar to the case 1) for $K \cup L$.
		\item[2)] Let $x \in \partial K\cap\operatorname{int}(L)$. The case when $x \in \operatorname{int}(K)\cap \partial L$ is treated similarly.
		\par
		Let $z = \{\lambda x: \lambda \geq 0\} \cap \partial L$.
		\begin{itemize}
			\item[2a)] Assume first that $z \in \partial L$ is parallel to $N_L(z)$. Consider the  convex cone that is the convex hull of the origin $o$ and the inner ball $B^n_2(z - r_L N_L(z), r_L)$ (see Figure \ref{Fig:r1}). Observe that this cone is contained in $L$. Using  similar triangles, we get
			$$
			\frac{\rho}{r_L} = \frac{||x||}{||z|| - r_L},
			$$ 
			where $||\cdot||$ denotes the Euclidean norm.
			Without loss of generality, we can assume that $r_L \leq \frac12 ||z||$ for any $z \in \partial L$. If $r_L = ||z||$ for some $z$, we choose $r_L = \frac12 ||z||$. Then
			$$
			\frac{||x||}{||z||} r_L \leq \rho = \frac{||x||}{||z|| - r_L} r_L \leq 2 \frac{||x||}{||z||} r_L.
			$$
			We put  $\rho_0 =  r_L \frac{\min\limits_{x \in \partial K} ||x||}{\max\limits_{z \in \partial L} ||z||}$. The set  $B^n_2(x, \rho_0) \cap B_2^n(x - r_K N_K(x), r_K)$
			 has non-empty interior. 
			Choosing $r = \min\{\frac{\rho_0}4, \frac{r_K}4\}$ gives that
			 $$
			 B_2^n(x - r N_K(x), r) \subseteq B_2^n(x - r_K N_K(x), r_K) \subseteq K
			 $$
			 and 
			$$
			B_2^n(x - r N_K(x), r) \subseteq B^n_2(x, \rho_0) \subseteq L.
			 $$
			 This implies that $B_2^n(x - r N_K(x), r) \subseteq K \cap L$.
			\begin{figure}[h!]
				\begin{center}
					\includegraphics[width=0.7\textwidth]{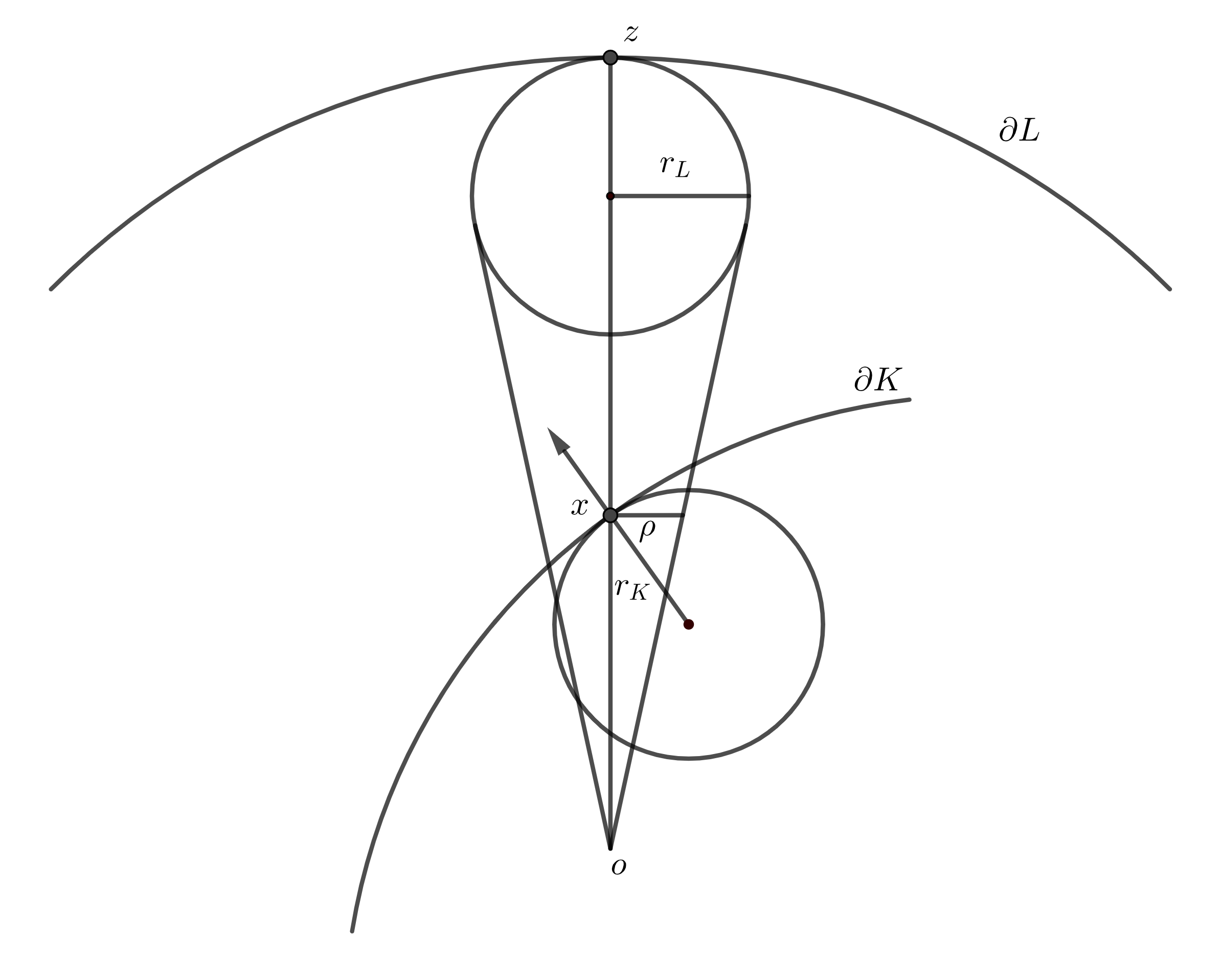}
					\caption{ Case 2a).}
					\label{Fig:r1}
				\end{center}
			\end{figure}
			\item[2b)] Now we treat the general case, when $z\in \partial L$ is not necessarily parallel to $N_L(z)$. 
			We reduce this case to the previous case.  As $x \in \text{int} (L)$,  the line segment $$[z_0, z] = \{\lambda x: \lambda \geq 0\} \cap  B_2^n(z - r_L N_L(z), r_L)$$ 
			is in the interior of $B_2^n(z - r_L N_L(z), r_L)$. Let $w$ be the midpoint of $[z_0, z]$ and let $\alpha=\dist (w, \partial (B_2^n(z - r_L N_L(z), r_L)))$ be the distance from $w$ to the boundary of $B_2^n(z - r_L N_L(z), r_L))$. Then 
			$$
			B_2^n(w, \alpha/2) \subset B_2^n(z - r_L N_L(z), r_L).
			$$
			Now we  consider the cone that is the convex hull of the origin $o$ and the  ball $B^n_2(w, \alpha/2)$.  We  then proceed as above, replacing  $z$ by $w+\alpha/2 \frac{x}{||x||}$, also noting that $||w+\alpha/2  \frac{x}{||x||}|| \leq ||z||$. 			
		\end{itemize} 
	\end{itemize}
	
	The proof that $K \cap L$ and $K \cup L$ admit an outer rolling ball is similar to the above, though slightly more technical.	
\end{proof}

\vskip 3mm
\noindent
\subsection{Special Cases} \label{SCEx}
\vskip 2mm 
The $L_p$-Steiner quermassintegrals generalize the known quermassintegrals, the dual mixed volumes and the mixed $L_p$ affine surface areas. Indeed,
\vskip 2mm
\noindent
(i) If $p=0$ and $K$ is $C^2_+$, 
\begin{equation}\label{p=0,1}
	\qv^0_k(K)  =n  \binom{n}{k} \, W_k(K).
\end{equation}

\noindent Then (\ref{p=0,1}) can be deduced immediately as follows.  From (\ref{Steiner as_p}) we get  that
\begin{equation}\label{p=0,2}
as_0(K+t B^n_2)=  \sum _{k=0}^\infty \qv^0_k(K)  \, t^k.
\end{equation}
On the other hand, by the classical Steiner formula, 
\begin{equation}\label{p=0,3}
as_0(K+t B^n_2)  = n \, \vol_n\left(K+t B^n_2\right) =  \sum _{k=0}^n \binom{n}{k} \,  \mathcal{W}_{k}(K) \, t^k.
\end{equation}
Comparing (\ref{p=0,2}) and (\ref{p=0,3}), we see that $\qv^0_k(K)=0$ for all $k>n$  and for $0 \leq k \leq n$ (\ref{p=0,1}) follows immediately.
\vskip 2mm
\noindent
(ii) If
$p = \pm \infty$ and $K$ is $C^2_+$,
we get from the definition that
$$
\qv^{\pm \infty}_k(K) = \binom{-n}{k} \widetilde{W}_{-k}(K^\circ) = (-1)^k \binom{n+k-1}{k} \widetilde{W}_{-k}(K^\circ),
$$
where $\widetilde{W}_{k}(K)$ are the dual quermassintegrals, \cite{Lu1}.
\vskip 3mm
\noindent
(iii) The $L_p$-Steiner quermassintegrals generalize the known {\it mixed $L_p$ affine surface areas }\cite{Lu, WY2010}.  We recall that 
for all $p \neq -n$ and all real $s$, the $s$-th mixed $L_p$ affine surface area of $K$ is defined  as 
\begin{equation*}
as_{p,\,s}(K) = \int\limits_{\partial K}  H_{n-1}(x)  ^{\frac{s +p}{n + p}} \langle x, N(x) \rangle ^{(1-p)\frac{n - s}{n + p} } d\mathcal{H}^{n-1}(x).
\end{equation*}
Then   we get for 
$k = l(n-1)$, $l \in \mathbb{N}$ and $p = 1$,
$$
\qv_{l(n-1)}^1(K)  = \qw_{l(n-1),\, l(n-1)}(K) = \binom{\frac n{n+1}}{l} as_{1,l(n+1)}(K).
$$
Similarly, for $m = 0$,
\begin{equation}\label{m=0}
\mathcal{W}^p_{0,\, k}(K) = \int\limits_{\partial K} \langle x, N(x) \rangle^{- k + \frac{n(1-p)}{n+p}} H_{n-1}^{\frac p{n+p}} \, d\mathcal{H}^{n-1}(x)  = as_{p+\frac kn (n+p),\, -k}(K).
\end{equation}
Note that the term $\mathcal{W}^p_{0,\, k}(K)$, which corresponds to the last summand in $\qv^p_k(K)$, is finite and positive for all $k$ and $p$ such that $\frac{p}{n + p}\geq 0$. However, for $p < 0$, it may become infinite (see Example 4.2).
When, in addition, $p = 0$ and $k = 1$, $\mathcal{W}^0_{0,\, 1}(K) $ is the surface area  $\vol_{n-1}(\partial K)$ of $K$.
\vskip 3mm
\noindent	
(iv) For the Euclidean unit ball $B^n_2$ and all  parameters $p \ne -n$ we get that 
	\begin{eqnarray} \label{unitball}
	\mathcal{V}^p_{k}(B^n_2) &=&  \sum\limits_{m = 0}^{k} \binom{\frac{n(1-p)}{n+p}}{ k-m} \mathcal{W}^p_{m,\, k}(B^n_2)\nonumber \\
	&=&
	\vol_{n-1}(\partial B^n_2) \, \sum\limits_{m = 0}^{k} \binom{\frac{n(1-p)}{n+p}}{k-m} \sum_{\substack{
				i_1, \dots, i_{n-1} \geq 0 \\
				i_1 + 2i_2 + \dots + (n-  1)i_{n-1}=m}}
		c(n, p, i(m))  \prod\limits_{j = 1}^{n - 1}  {n - 1\choose j}^{i_j}.
\end{eqnarray} 
We define
\begin{equation}\label{cpk}
C(n,p,k) = 	\sum\limits_{m = 0}^{k} \binom{\frac{n(1-p)}{n+p}}{k-m} \sum_{\substack{
				i_1, \dots, i_{n-1} \geq 0 \\
				i_1 + 2i_2 + \dots + (n-  1)i_{n-1}=m}}
		c(n, p, i(m))  \prod\limits_{j = 1}^{n - 1}  {n - 1\choose j}^{i_j}.
\end{equation} 
Then, 
\begin{equation}\label{Kugel0}
\mathcal{V}^p_{k}(B^n_2)=   C(n,p,k) \, \vol_{n-1}(\partial B^n_2).
\end{equation}
In particular, 
\begin{equation}\label{c1k}
C(n,1,k) =  \sum_{\substack{
				i_1, \dots, i_{n-1} \geq 0 \\
				i_1 + 2i_2 + \dots + (n-  1)i_{n-1}=k}}
		c(n, 1, i(k))  \prod\limits_{j = 1}^{n - 1}  {n - 1\choose j}^{i_j}
\end{equation} 
and 
\begin{equation}\label{Kugel}
\mathcal{W}_{k,\, k}(B^n_2)=  \vol_{n-1}(\partial B^n_2) \, C(n,1,k).
\end{equation}
\vskip 2mm

\vskip 3mm
\noindent
(v) Polytopes $P$ are not in $\mathcal{C}_p$ but we would like to note that for polytopes the $L_p$-Steiner quermassintegrals exhibit similar properties as the $L_p$ affine surface areas.    By (\ref{m=0}), 
$$
\qw_{0, \, k}^0(P) = \int\limits_{\partial P} \langle x, N(x) \rangle^{- k + 1}  \, d\mathcal{H}^{n-1}(x)  = as_{k,\, -k}(P), 
$$
and by definition,
\begin{equation} \label{V0-poly}
\mathcal{V}^0_{k}(P) =  \binom{1}{k} \,  \qw_{0, \, k}^0(P) = 
\left\{
\begin{aligned} & n\,  \vol_n(P)  \quad &\text{ if } \, k=0 \\ 
&   \vol_{n-1}(\partial P) \quad &\text{ if } \,  k=1 \\
& 0 \quad &\text{ if }\, \,  \,  k \geq 2.
\end{aligned}\right.
\end{equation}
For $p \neq 0$, we get for all $k \geq 0$,
\begin{equation} \label{Vp-poly}
\mathcal{V}^p_{k}(P) = 0 \,  \,  \text{ if }\, \,  0 < p \leq \infty \, \text{ and }  \,  -\infty \leq p < -n.
\end{equation}

\vskip 3mm
\noindent

\subsection{The expressions $C(n,p,k)$}

The $L_p$ affine surface area is homogeneous of degree $n \frac{n-p}{n+p}$,  i.e.,  
\begin{equation}\label{homo}
as_p( t K) = t^{n \frac{n-p}{n+p}} as_p(K).
\end{equation}
 Hence
\begin{equation} \label{comb-formula}
	as_p((1+t) B^n_2) = (1 + t)^{n \frac{n-p}{n+p}} \vol_{n-1}(\partial B^n_2) = \vol_{n-1}(\partial B^n_2) \sum_{k=0}^{\infty} {n \frac{n-p}{n+p}\choose k} \ t^k.
\end{equation}	
Observe that for $p=n$, $as_n((1+t) B^n_2) = \vol_{n-1}(\partial B^n_2)$  and that 
on the right hand side of (\ref{comb-formula}) ${0 \choose k} =0$ for all $k$,  except for $k=0$, when ${0 \choose 0} =1$.
In particular, this  means that 
\begin{equation*} 
C(n,n,k) =\left\{
\begin{aligned}
&1 \quad &\text{ if }\  k=0,\\
&0 \quad &\text{ if }\,  k \geq 1.
\end{aligned}\right.
\end{equation*}
\par
\noindent
For $p \neq n,  -n$,  applying \eqref{Steiner as_p} when $K$ is the Euclidean unit ball $B^n_2$, we get
$$
as_p(B^n_2 + tB^n_2) = \sum_{k = 0}^{\infty} \mathcal{V}_k^p(B^n_2) \, t^k.
$$
	Therefore, 
	$$
	\mathcal{V}_k^p(B^n_2) = \vol_{n-1}(\partial B^n_2) {n \frac{n-p}{n+p}\choose k}
	$$
	which implies that
	\begin{equation} \label{cpk-Binomial}
		C(n, p, k) = {n \frac{n-p}{n+p}\choose k}.
	\end{equation}
\par
\noindent	
	This observation results in the following  combinatorial formula.
\par
\noindent
\begin{cor} Let $p \in \mathbb{R}$, $p \neq -n$. Then
	\begin{equation*}
	{n \frac{n-p}{n+p}\choose k} = \sum\limits_{m = 0}^{k} \binom{\frac{n(1-p)}{n+p}}{k-m} \sum_{\substack{
			i_1, \dots, i_{n-1} \geq 0 \\
			i_1 + 2i_2 + \dots + (n-  1)i_{n-1}=m}}
	c(n, p, i(m))  \prod\limits_{j = 1}^{n - 1}  {n - 1\choose j}^{i_j}.
	\end{equation*}
\end{cor}
\vskip 3mm
\noindent
We summarize some consequences for the Euclidean ball in the next corollary.
\par
\noindent
\begin{cor} \label{Kugel0+n}
 \begin{equation*} 
(i)  \hskip 4mm   \mathcal{V}^0_{k}(B^n_2) = {n \choose k} \vol_{n-1}(\partial B^n_2) =  \left\{
\begin{aligned}
& {n \choose k} \vol_{n-1}(\partial B^n_2) \quad &\text{ if }\  0 \leq k \leq n,\\
& 0 \quad &\text{ if }\,  k \geq n+1.
\end{aligned}\right.
\end{equation*}
\par
\begin{equation*} 
(ii) \hskip 4mm  \mathcal{V}^n_{k}(B^n_2) = C(n,n,k)   \, \vol_{n-1}(\partial B^n_2) = \left\{
\begin{aligned}
& \vol_{n-1}(\partial B^n_2)\quad &\text{ if }\  k=0,\\
& 0 \quad &\text{ if }\,  k \geq 1.
\end{aligned}\right.
\end{equation*}
\par
\begin{equation*} 
\hskip -50mm (iii) \hskip 4mm  
	\mathcal{V}_k^p(B^n_2) = \vol_{n-1}(\partial B^n_2) {n \frac{n-p}{n+p}\choose k}.
\end{equation*}	

\end{cor}

\vskip 3mm
\noindent
We now analyze the expressions $C(n,p,k)$ further.  
This is also needed to determine  continuity issues of $L_p$-Steiner quermassintegrals in Section \ref{continuity}.
We start with the case $p=1$.
\noindent
\begin{prop} \label{Ckk}
	The following holds for all $n \geq 2$:
	\begin{itemize}
		\item[(i)] $C(n,1,k) > 0$ for $k \leq n-1$;
		\item[(ii)] $(-1)^{k - n+1}C(n,1,k) > 0$  for $k \geq n$.
	\end{itemize}
\end{prop}
\par
\noindent
\begin{proof}
	\begin{itemize}
		\item[(i)] Denote $\alpha = n\frac{n-1}{n+1}$. Using \eqref{cpk-Binomial}, 
		\begin{equation}\label{cn1k-binom}
		C(n, 1, k) = \binom{\alpha}{k} = \frac{\alpha (\alpha - 1)(\alpha - 2) \cdots (\alpha - k +1)}{k(k-1)\cdots 1}.
		\end{equation}
		It is clear that the denominator is positive. We show that the numerator is also positive when $k \leq n-1$.
		Note that $\alpha - j = n-2 + \frac2{n+1} - j$ for $1 \leq j \leq k-1$. Then $\alpha - j> 0$ when $j \leq n - 2$. Since $k \leq n - 1$, $j$ can change in the interval $1 \leq j\leq n-2$ which implies that  $C(n,1,k) >0.$
		\item[(ii)] The numerator of \eqref{cn1k-binom} consists of $k - 1$ factors $\alpha - j$ each of which is positive when $j \leq n-2$ and negative when $j \geq n-1.$ If $k \geq n$, the first $n-2$ terms of \eqref{cn1k-binom} are positive, and the remaining $k - n+1$ terms are negative which implies the statement.
	\end{itemize}
\end{proof}
\noindent
Now we address the general case.
\par
%\noindent
\begin{prop} \label{Cnpk}
	The following holds for all $n \geq 2$:
	\begin{itemize}
		\item[(i)] If $p < -n$, then  $(-1)^k C(n,p,k) > 0$ for $k \in \mathbb{N}$;
		\item[(ii)] If $-n < p < n - 2 + \frac2{n+1}$,  then 
		$C(n,p,k) > 0$ for all $k \leq \lfloor n -2p +\frac{2p^2}{n+p} \rfloor + 1$
		and 
		$(-1)^{k -( \lfloor n -2p +\frac{2p^2}{n+p} \rfloor + 1)}C(n,p,k)~>~0$  for $k > \lfloor n -2p +\frac{2p^2}{n+p} \rfloor + 1$;
		\item[(iii)] If $ n - 2 + \frac2{n+1} \leq p < n$, then  $(-1)^{k-1} C(n,p,k) > 0$ for $k \in \mathbb{N}$;
		\item[(iv)] If $p = n$, then  $C(n, p, 0) = 1$ and $C(n,p,k) = 0$ for $k \in \mathbb{N}$;
		\item[(v)] If $p > n$, then $(-1)^k C(n,p,k) > 0$ for $k \in \mathbb{N}$.
	\end{itemize}
\end{prop}
\vskip 2mm
\noindent	
\begin{proof}
	Denote $\alpha = n\frac{n-p}{n+p}$. By \eqref{cpk-Binomial}, 
	\begin{equation}\label{cnpk-binom}
	C(n, p, k) = \binom{\alpha}{k} = \frac{\alpha (\alpha - 1)(\alpha - 2) \cdots (\alpha - k +1)}{k(k-1)\cdots 1}.
	\end{equation}
	\begin{itemize}
		\item[(i)] 
		Since denominator of \eqref{cnpk-binom} is positive, numerator of \eqref{cnpk-binom} determines the sign of $C(n, p ,k)$.
		If $p < -n$, then $\alpha < 0$  which implies that $\alpha - j < 0$ for all $1\leq j \leq k-1$. As there are $k$ negative terms in the numerator, we get the result.
		\item[(ii)] Here $\alpha - j > 0$ when $ j \leq  \lfloor n -2p +\frac{2p^2}{n+p} \rfloor$, and $\alpha - j < 0$ otherwise. Since $1 \leq j \leq  k-1$, if $k \leq \lfloor n -2p +\frac{2p^2}{n+p} \rfloor + 1$, then $\alpha - j > 0$ which yields that $C(n, p, k) > 0$.  If $k > \lfloor n -2p +\frac{2p^2}{n+p} \rfloor + 1$, then the first $\lfloor n -2p +\frac{2p^2}{n+p} \rfloor$ terms $\alpha - j$ are positive and the remaining $k - (\lfloor n -2p +\frac{2p^2}{n+p} \rfloor + 1)$ terms are negative which implies the result.

		\item[(iv)]  If $ p = n$, then $\alpha = 0$,  $C(n, p, 0) = \binom{0}{0} = 1$ and $C(n, p, k) = \binom{0}{k} = 0$ for $k \in \mathbb{N}$.

	\end{itemize}	
\noindent 
The proof of {\it (iii)} and {\it (v)}  is similar to proof of {\it (i)}.
\end{proof}
\vskip 2mm
\noindent
{\bf Remark}
\vskip 2mm
\noindent
In the special case when $p = \frac{n (n-l)}{n + l}$, $l \in \mathbb{N}$,  $\frac{n(n-p)}{n+p}$ is an integer.  This simplifies the binomial coefficients $\binom{\frac{n(n-p)}{n+p}}{k}$ that appear in \eqref{p-SC1}. Thus, the coefficients $C(n, p, k) > 0$ for $k \leq l$ and $C(n, p ,k) = 0$ for $k > l$.
\vskip 3mm

\section{Properties of the $L_p$-Steiner quermassintegrals}

\subsection{Homogeneity and Invariance}

	\begin{theorem}\label{VProperties}
	Let $K$ be a convex body in $\mathcal{C}^n$ and let $p \in\mathbb{R}$ be such that $ p\geq 0$ or $p < -n$. Then for all $ k \in \N$, we have
	\begin{itemize}
		\item[(i)] $\mathcal{W}^p_{m,\, k}(K)$  and $\mathcal{V}^p_k (K)$ are homogeneous of degree $n \frac{n - p}{n+p} - k$.
		\item[(ii)] $\mathcal{W}^p_{m,\, k}(K)$ and $\mathcal{V}^p_k (K)$ are invariant  under rotations and reflections.
			\end{itemize}		
	\end{theorem}
\par
%\vskip 2mm
\noindent
{\bf Remark}
\vskip 2mm
\noindent
Property (i) of Theorem \ref{VProperties}  implies that for the Euclidean ball with radius $r$, 
\begin{equation} \label{hom-ball}
	\mathcal{V}^p_{k}(r B^n_2)= r^{n \frac{n - p}{n+p} - k} \, \mathcal{V}^p_{k}(B^n_2) = r^{n \frac{n - p}{n+p} - k} \,  \vol_{n-1}(\partial B^n_2) \, C(n,p,k).
\end{equation}
\vskip 2mm
\noindent	

\vskip 4mm
\noindent
The following corollary, for the case $p=1$, immediately follows from Theorem \ref{VProperties}. In addition to rotation- and reflection-invariance, we  also have invariance under translations. 	
\begin{cor}
	Let $K$ be a convex body in $\mathcal{C}^n$. Then for all $ k \in \N$, 
	$\qv^1_k = \qw_{k, \, k}$ are invariant under rigid motions and  homogeneous of degree $n \frac{n - 1}{n+1} - k$.
	\par
	\noindent
	\end{cor}

\subsubsection{Proof of Theorem~\ref{VProperties}}
\vskip 2mm
\noindent
To show Theorem ~\ref{VProperties}, we need the following facts (see, e.g., \cite{SW2004}).
%\par
%\noindent
\begin{prop}\label{lin_map}
	Let $g:\partial K\rightarrow
	\mathbb R$ be an integrable function, and $T:\mathbb R^{n}\rightarrow
	\mathbb R^{n}$ be an invertible, linear map. Then
	$$
	\int_{\partial K} g(x)\,d\mathcal{H}^{n-1}(x)
	=|\det(T)|^{-1}\, 
	\int_{\partial T(K)}\|T^{-1t}(N_{ K}(T^{-1}(y)))\|^{-1}
	g(T^{-1}(y)) \, d\mathcal{H}^{n-1}(x)
	$$
	and
	$$
	\langle T^{-1}(y,) N_K(T^{-1}y) \rangle = \langle y, N_{T(K)}(y) \rangle \, \|T^{-1t}(N_{ K}(T^{-1}(y)))\|
	$$
	for all $y \in \partial T(K)$.
\end{prop}
\medskip

\noindent We apply Proposition~\ref{lin_map} to $\mathcal{W}^p_{m,\, k}(K) $ and get  
\begin{eqnarray} \label{invariance}
&&\mathcal{W}^p_{m,\, k}(K) = \nonumber \\
&&  |\det(T)|^{-1} \, \int\limits_{\partial T(K)} \|T^{-1t}(N_{ K}(T^{-1}(y)))\|^{m - k -1+ \frac{n(1-p)}{n+p}} \langle y, N_{T(K)}(y) \rangle^{m - k + \frac{n(1-p)}{n+p}} H_{n-1}^{\frac p{n+p}}  (T^{-1}(y)) \nonumber \\&&\sum_{\substack{
			i_1, \dots, i_{n-1} \geq 0 \\
			i_1 + 2i_2 + \dots + (n-  1)i_{n-1}=m}}
	 c(n, p, i(m))  \prod\limits_{j = 1}^{n - 1}  {n - 1\choose j}^{i_j} H_{j}^{i_j}\left(T^{-1}(y) \right) \, d\mathcal{H}^{n-1}(y).
\end{eqnarray}	
\vskip 2mm
\noindent	
{\it (i)} Let $a \in \mathbb{R}$.   We apply (\ref{invariance}) to $T= a \, Id$. We also use that  for $y \in \partial (aK)$, 
		$$H_j(y) = \frac{H_j(T^{-1}y)}{a^j}.$$ 
	
\noindent We get	
\begin{eqnarray*}
&&\mathcal{W}^p_{m,\, k}(K) = a^{k-n\frac{n-p}{n+p}} \,  \\
 && \int\limits_{\partial aK} \langle y, N_{aK}(y) \rangle^{m - k + \frac{n(1-p)}{n+p}} H_{n-1}^{\frac p{n+p}}  (y)\sum_{\substack{
			i_1, \dots, i_{n-1} \geq 0 \\
			i_1 + 2i_2 + \dots + (n-  1)i_{n-1}=m}}
	 c(n, p, i(m))  \prod\limits_{j = 1}^{n - 1}  {n - 1\choose j}^{i_j} H_{j}^{i_j}\left(y \right) \, d\mathcal{H}^{n-1}(y) \\
	 &&= a^{k-n\frac{n-p}{n+p}} \,  \mathcal{W}^p_{m,\, k}(a K).
\end{eqnarray*}

\noindent Consequently,
$$
\mathcal{V}^p_{ k}(K) =  \sum\limits_{m = 0}^{k} \binom{\frac{n(1-p)}{n+p}}{k-m} \mathcal{W}^p_{m,\, k}(K)  = a^{k-n\frac{n-p}{n+p}} \,   \sum\limits_{m = 0}^{k} \binom{\frac{n(1-p)}{n+p}}{k-m} \mathcal{W}^p_{m,\, k}(a K) = a^{k-n\frac{n-p}{n+p}} \, \mathcal{V}^p_{k}(a K).
$$
\vskip 2mm
\noindent		
{\it (ii)}
If $T$ is a rotation or a reflection, then $|\det T|=1$, $\|T^{-1t}(N_{ K}(T^{-1}(y)))\| = \| N_{ K}(T^{-1}(y))\| = 1$ and 
for all $1 \leq j \leq n-1$, 
$$\{ H_{j}(y): y \in \partial T(K) \}= \{H_{j}\left(x \right): \,  x \in \partial K \}.$$
Thus 	
\begin{eqnarray*}
&&\mathcal{W}^p_{m,\, k}(K) =\\
 && \int\limits_{\partial T(K)} \langle y, N_{T(K)}(y) \rangle^{m - k + \frac{n(1-p)}{n+p}} H_{n-1}^{\frac p{n+p}}  (y)\sum_{\substack{
			i_1, \dots, i_{n-1} \geq 0 \\
			i_1 + 2i_2 + \dots + (n-  1)i_{n-1}=m}}
	 c(n, p, i(m))  \prod\limits_{j = 1}^{n - 1}  {n - 1\choose j}^{i_j} H_{j}^{i_j}\left(y \right) \, d\mathcal{H}^{n-1}(y) \\
	 &&= \mathcal{W}^p_{m,\, k}(T(K)),
\end{eqnarray*}
which implies
$$
\mathcal{V}^p_{ k}(K) =  \sum\limits_{m = 0}^{k} \binom{\frac{n(1-p)}{n+p}}{k-m} \mathcal{W}^p_{m,\, k}(K)  =  \sum\limits_{m = 0}^{k} \binom{\frac{n(1-p)}{n+p}}{k-m} \mathcal{W}^p_{m,\, k}(T(K)) = \mathcal{V}^p_{k}(T(K)).
$$
\vskip 2mm

\noindent 

\vskip 4mm
\noindent
{\bf Remarks}
\vskip 2mm
\noindent
1. Under the additional assumption that $K$ is $C^2_+$, the proof of the homogeneity property is an immediate consequence of 
the  homogeneity of $L_p$ affine surface area and \eqref{Steiner as_p}. Indeed,  we have
		$$
		as_p(\lambda K +t B^n_2) = as_p \left(\lambda \left(K + \frac t\lambda B^n_2\right) \right) = \lambda^{n\frac{n-p}{n+p}} as_p \left(K + \frac t\lambda B^n_2 \right) = \sum\limits_{k = 0}^\infty \qv^p_{k}(K) \lambda^{n\frac{n-p}{n+p} - k} t^k.
		$$ 
		On the other hand, using \eqref{Steiner as_p} directly, we get
		$$
		as_p(\lambda K +t B) =\sum\limits_{k = 0}^\infty \qv^p_{k}( \lambda K) t^k.
		$$
		Therefore,
		$$
		\qv^p_{k}( \lambda K) = \lambda^{n\frac{n-p}{n+p} - k} \qv^p_{k}(K).
		$$ 
\vskip 2mm
\noindent
2. We cannot expect that $\mathcal{V}^p_{k}(K), \mathcal{W}^p_{m,\, k}(K)$ are invariant under general linear transformations. 
For instance, by (\ref{invariance}), and also using (see \cite{SW2004})  that 
$$
H_{n-1}(x)^{}
=\|T^{-1t}(N_{ K}(x))\|^{n+1}|\det(T)|^{2}
H_{n-1}(T(x))^{}
$$
we have
\begin{eqnarray*}
\mathcal{V}^p_{1}(K)  &=& \frac{n(1-p)}{n+p} \mathcal{W}^p_{0,\, 1}(K) +   \mathcal{W}^p_{1,\, 1}(K)  =  \frac{n}{n+p} |\det(T)|^{-\frac{n-p}{n+p}} \, \\
&&\Bigg ((1-p) \, \int \limits_{\partial T(K)} \|T^{-1t}(N_{ K}(T^{-1}(y)))\|^{-1} \langle y, N_{T(K)}(y) \rangle^{-1 + \frac{n(1-p)}{n+p}} H_{n-1}^{\frac p{n+p}} (y)\,  d\mathcal{H}^{n-1}(y) \\
 &+& 
(n-1)  \int\limits_{\partial T(K)}  \langle y, N_{T(K)}(y) \rangle^{ \frac{n(1-p)}{n+p}} H_{n-1}^{\frac p{n+p}}  (y)  H_{1}(T^{-1}(y)) \, 
d\mathcal{H}^{n-1}(y) \Bigg)
\end{eqnarray*}
In particular, if $n=2$, then $H_{n-1} =H_{1}$ and thus 
\begin{eqnarray*}
\mathcal{V}^p_{1}(K)  &=& \frac{2(1-p)}{2+p} \mathcal{W}^p_{0,\, 1}(K) + \mathcal{W}^p_{1,\, 1}(K) = \\
 && \hskip -11mm  \frac{2}{2+p} |\det(T)|^{-\frac{2-p}{2+p}} \Bigg ( 
(1-p)\, \int \limits_{\partial T(K)} \|T^{-1t}(N_{ K}(T^{-1}(y)))\|^{-1} \langle y, N_{T(K)}(y) \rangle^{-1 + \frac{2(1-p)}{2+p}} H_{1}^{\frac 2{2+p}} (y) \, d\mathcal{H}^{1}(y) \\
 && \hskip -11mm +
|\det(T)|^{2} \int\limits_{\partial T(K)} \|T^{-1t}(N_{ K}(T^{-1}(y)))\|^{3} \langle y, N_{T(K)}(y) \rangle^{ \frac{2(1-p)}{2+p}} H_{1}^{\frac {2(p+1)}{2+p}}(y) \, d\mathcal{H}^{1}(y) \Bigg),
\end{eqnarray*}
while 
\begin{eqnarray*}
&&\mathcal{V}^p_{1}(T(K) ) = \frac{2(1-p)}{2+p} \mathcal{W}^p_{0,\, 1}(T(K))+  \mathcal{W}^p_{1,\, 1}(T(K)) =\\
&& \hskip -10mm \frac{2}{2+p} \Bigg( (1-p) \int \limits_{\partial T(K)}  \langle y, N(y) \rangle^{-1 + \frac{2(1-p)}{2+p}} H_{1}^{\frac 2{2+p}} (y) \, d\mathcal{H}^{1}(y) 
  +  \int \limits_{\partial T(K)}  \langle y, N(y) \rangle^{ \frac{2(1-p)}{2+p}} H_{1}^{\frac {2(p+1)}{2+p}} (y) \, d\mathcal{H}^{1}(y) 
\Bigg).
\end{eqnarray*}
That is, unless $T$ is an isometry, we cannot expect to have invariance.

	\subsection{Valuation property}
		
		For $p = 1$, we have the following		
		\begin{theorem}\label{WkkProperties}
			For all $ k \in \N$, $\qw_{k, \, k}$ are  valuations on the set~$\mathcal{C}^n$. 
			\end{theorem}
		\noindent More generally,  we obtain 
		\begin{theorem}\label{WpmkProperties}
			Let $p \in \mathbb{R}$ be such that $ p\geq 0$ or $p < -n$. For all $ k, m\in \N$, $\qw_{m, \, k}^p$  and $\qv^p_{k}$ are  valuations on the set~$\mathcal{C}^n$.
		\end{theorem}

%%%%%%%%%%%%%%%%%%%%%%%%%%%%%%%%%%%%%%%%%%%%%%%
%%%%%%%%%%%%%%%%%%%%%%%%%%%%%%%%%%%%%%%%%%%%%%%%%

\subsubsection{Proof of Theorem~\ref{WkkProperties} and Theorem~\ref{WpmkProperties}}\label{ValProof}

We need the following lemma which can be found in e.g. \cite{Sc}.

\begin{lemma}\label{CurvUnionBod}
	Let $C$ and $K$ be convex bodies in $\mathbb R^{n}$ and suppose that
	$C\cup K$ is a convex body. Then we have for all $x\in\partial C\cap\partial K$ and for all $\alpha \geq 0$ where  the principal curvatures $k_{j} (\partial(C\cup K),x)$, $k_{j}(\partial(C\cap K),x)$,
	$k_{j}(C,x)$ and $k_{j}(K,x)$ exist for all $1 \leq j \leq n-1$,
	\begin{eqnarray*}
		H_{n-1}(\partial(C\cup K),x)^{\frac{1}{n+1}} k_j (\partial(K \cup C), x)^\alpha =\min\{H_{n-1}(C,x)^{\frac{1}{n+1}} k_j (C, x)^\alpha, H_{n-1}(K,x) ^{\frac{1}{n+1}} k_j (K, x)^\alpha\}
	\end{eqnarray*}
	and
	\begin{eqnarray*}
		H_{n-1}(\partial(C\cap K),x)^{\frac{1}{n+1}} k_j (\partial(K \cap C), x)^\alpha =\max\{H_{n-1}(C,x)^{\frac{1}{n+1}} k_j (C, x)^\alpha, H_{n-1}(K,x) ^{\frac{1}{n+1}} k_j (K, x)^\alpha\}.
	\end{eqnarray*}
	\vskip 2mm
	\noindent

%%%%%%%%%%%%%%%%%%%%%%%%%%%%%%%%%%%%%%%%%%%%%%%%

\noindent Moreover, for all $1 \leq i_1,\dots, i_{j} \leq n-1$, and  for all $\alpha_1, \alpha_2, \dots,  \alpha_j \geq 0$, 
\begin{eqnarray*}
	&&H_{n-1}(\partial(C\cup K),x)^{\frac{1}{n+1}} \prod_{i=1}^j k_{i_j} (\partial(K \cup C), x)^{\alpha_j} =\\
	&&\min\{H_{n-1}(C,x)^{\frac{1}{n+1}} \prod_{i=1}^j k_{i_j} (C, x)^{\alpha_j} , H_{n-1}(K,x) ^{\frac{1}{n+1}} \prod_{i=1}^j k_{i_j} (K, x)^{\alpha_j}\}
\end{eqnarray*}
and
\begin{eqnarray*}
	&&H_{n-1}(\partial(C\cap K),x)^{\frac{1}{n+1}} \prod_{i=1}^j k_{i_j} (\partial(K \cap C), x)^{\alpha_j} =\\
	&&\max\{H_{n-1}(C,x)^{\frac{1}{n+1}} \prod_{i=1}^j k_{i_j} (C, x)^{\alpha_j} , H_{n-1}(K,x) ^{\frac{1}{n+1}} \prod_{i=1}^j k_{i_j} (K, x)^{\alpha_j}\}.
\end{eqnarray*}
\end{lemma}
\vskip 3mm

\vskip 2mm

\begin{theorem}\label{WkkVal}
	\par
	\noindent
	Let $p \in \mathbb{R}$ be such that $ p\geq 0$ or $p < -n$. For all $1 \leq i_1,\dots, i_{j} \leq n-1$ and $\alpha_1, \alpha_2, \dots,  \alpha_j~\geq~0$, 
	$$\int\limits_{\partial K}  H_{n-1}^{\frac 1{n+1}}(x)  \
	\prod_{i=1}^j k_{i_j}^{\alpha_j} (x) \, d\mathcal{H}^{n-1}(x)$$ 
	are valuations on the set~$\mathcal{C}^n$.
\end{theorem}
\vskip 2mm
\noindent
\begin{proof}
	\noindent Since  $C$ and $K$ are in $\mathcal{C}^n$, then $C\cap K$ and $C \cup K$ are in $\mathcal{C}^n$ by Proposition~\ref{int_union}.
	To prove the valuation property, we follow the approach of \cite{Sc}.
	Let $C$ and $K$ be convex bodies in $\mathbb R^{n}$ such that
	$C\cup K$ is a convex body. As above, we decompose
	\begin{eqnarray*}
		\partial(C\cup K)
		&=&\{\partial C\cap\partial K\}\cup\{\partial C\cap K^{c}\}\cup\{C^{c}\cap \partial K\}   \\
		\partial(C\cap K)
		&=&\{\partial C\cap\partial K\}
		\cup\{\partial C\cap\operatorname{int}(K)\}\cup\{\operatorname{int}(C)\cap \partial K\}    \\
		\partial C&=&\{\partial C\cap\partial K\}\cup\{\partial C\cap K^{c}\}
		\cup\{\partial C\cap\operatorname{int}(K)\}   \\
		\partial K&=&\{\partial C\cap\partial K\}\cup\{\partial K\cap C^{c}\}
		\cup\{\partial K\cap\operatorname{int}(C)\}.
	\end{eqnarray*}
	Since all the sets (except possibly $\partial C\cap\partial K$) in the right-hand side of the above decomposition are open subsets of $\partial C$, $\partial K$, $\partial (C\cap K)$
	and $\partial(C\cup K)$,  the integrals over those sets cancel out. Below we use $d\mathcal{H}^{n-1}_K(x)$ to denote the usual surface area measure on $\partial K$ and $k_j(K,x)$ to emphasize that this is the principal curvature of $K$ at point $x$. It remains to show
	\begin{eqnarray*}
		&& \int_{\partial C\cap\partial K}H_{n-1}(\partial(C\cup K),x)^{\frac{1}{n+1}} \prod_{i=1}^j k_{i_j}(\partial(K \cup C), x)^{\alpha_j} d\mathcal{H}^{n-1}_{C\cup K}(x) \\
		&& +\int_{\partial C\cap\partial K} H_{n-1}(\partial(C\cap K),x)^{\frac{1}{n+1}} \prod_{i=1}^j k_{i_j}(\partial(K \cap C), x)^{\alpha_j} d\mathcal{H}^{n-1}_{C\cap K}(x)\\
		&&=\int_{\partial C\cap\partial K} H_{n-1}(C,x)^{\frac{1}{n+1}} \prod_{i=1}^j k_{i_j}(C, x)^{\alpha_j} d\mathcal{H}^{n-1}_{C}(x)
		+\int_{\partial C\cap\partial K} H_{n-1}(K,x)^{\frac{1}{n+1}} \prod_{i=1}^j k_{i_j}(K, x)^{\alpha_j} d\mathcal{H}^{n-1}_{K}(x).
	\end{eqnarray*}
	Please note that $d\mathcal{H}^{n-1}_{C\cup K}=d\mathcal{H}^{n-1}_{C\cap K}$ on $\partial K\cap \partial C$. 
	This holds because both measure are equal to the $(n-1)$-dimensional
	Hausdorff measure. Therefore, it is left to show
	\begin{eqnarray*}
		&& \int_{\partial C\cap\partial K}H_{n-1}(\partial(C\cup K),x)^{\frac{1}{n+1}} \prod_{i=1}^j k_{i_j}(\partial(K \cup C), x)^{\alpha_j} d\mathcal{H}^{n-1}_{C\cap K}(x) \\
		&&  +\int_{\partial C\cap\partial K} H_{n-1}(\partial(C\cap K),x)^{\frac{1}{n+1}} \prod_{i=1}^j k_{i_j}(\partial(K \cap C), x)^{\alpha_j} d\mathcal{H}^{n-1}_{C\cap K}(x)\\
		&& =\int_{\partial C\cap\partial K} H_{n-1}(C,x)^{\frac{1}{n+1}} \prod_{i=1}^j k_{i_j}(C, x)^{\alpha_j} d\mathcal{H}^{n-1}_{C}(x)
		+\int_{\partial C\cap\partial K} H_{n-1}(K,x)^{\frac{1}{n+1}} \prod_{i=1}^j k_{i_j}(K, x)^{\alpha_j} d\mathcal{H}^{n-1}_{K}(x).
	\end{eqnarray*}
	By Lemma \ref{CurvUnionBod} this is equivalent to
	\begin{eqnarray*}
		&&\int_{\partial C\cap\partial K}\min\{H_{n-1}(C,x)^{\frac{1}{n+1}} \prod_{i=1}^j k_{i_j} (C, x)^{\alpha_j} , H_{n-1}(K,x) ^{\frac{1}{n+1}} \prod_{i=1}^j k_{i_j} (K, x)^{\alpha_j}\} d\mathcal{H}^{n-1}_{C\cap K}(x)\\
		&& +\int_{\partial C\cap\partial K}\max\{H_{n-1}(C,x)^{\frac{1}{n+1}} \prod_{i=1}^j k_{i_j} (C, x)^{\alpha_j} , H_{n-1}(K,x) ^{\frac{1}{n+1}} \prod_{i=1}^j k_{i_j} (K, x)^{\alpha_j}\} 
		d\mathcal{H}^{n-1}_{C\cap K}(x)   \\
		&&=\int_{\partial C\cap\partial K} H_{n-1}(C,x)^{\frac{1}{n+1}} \prod_{i=1}^j k_{i_j}(C, x)^{\alpha_j} d\mathcal{H}^{n-1}_{C}(x)
		+\int_{\partial C\cap\partial K} H_{n-1}(K,x)^{\frac{1}{n+1}} \prod_{i=1}^j k_{i_j}(K, x)^{\alpha_j} d\mathcal{H}^{n-1}_{K}(x).
	\end{eqnarray*}
	This holds since for any real numbers $a$ and $b$, we have
	$$
	a+b=\min\{a,b\}+\max\{a,b\}.
	$$

	\vskip 2mm
\end{proof}
\vskip 3mm
\noindent
We note that in the setting of Lemma \ref{CurvUnionBod}, $N_K(x) = N_C(x) =N_{C \cap K}(x) = N_{C\cup K}(x)$. Using this observation together with Lemma \ref{CurvUnionBod} and the decomposition in the proof of Theorem \ref{WkkVal}, we get the following generalization.
\begin{theorem} \label{WpmkVal}
	\noindent
	Let $p \in \mathbb{R}$ be such that $ p\geq 0$ or $p < -n$.  For all $1 \leq i_1,\dots, i_{j} \leq n-1$ and $\alpha_1, \alpha_2, \dots,  \alpha_j~\geq~0$, 
	$$\int\limits_{\partial K} \langle x, N(x) \rangle^{m - k + \frac{n(1-p)}{n+p}}  H_{n-1}^{\frac p{n+p}}(x)  \
	\prod_{i=1}^j k_{i_j}^{\alpha_j} (x) \, d\mathcal{H}^{n-1}(x)$$ 
	are valuations on the set~$\mathcal{C}^n$.
\end{theorem}

\vskip 3mm

\noindent
{\bf Proof of Theorem~\ref{WkkProperties}} \ \ 
\noindent The result follows immediately from Theorem~\ref{WkkVal} as 
\begin{equation*}
\mathcal{W}_{k,\, k}(K) = \int\limits_{\partial K}  H_{n-1}^{\frac 1{n+1}}(x)  \sum_{\substack{
		i_1, \dots, i_{n-1} \geq 0 \\
		i_1 + 2i_2 + \dots + (n-  1)i_{n-1}=k}}
c(n, 1, i(k))  \prod\limits_{j = 1}^{n - 1}  {n - 1\choose j}^{i_j} H_{j}^{i_j}\left(x \right) \, d\mathcal{H}^{n-1}(x)
\end{equation*}
is a sum (up to constants) of integrals of the form 
$$
\int\limits_{\partial K}  H_{n-1}^{\frac 1{n+1}}(x)  \
\prod_{i=1}^j k_{i_j}^{\alpha_j} (x) \, d\mathcal{H}^{n-1}(x)
$$
and as the sum of valuations, it is again a valuation.

\vskip 2mm

\noindent{\bf Proof of Theorem~\ref{WpmkProperties}} 
\noindent The result follows immediately from Theorem~\ref{WpmkVal} and the fact that the linear combination of valuations is again a valuation.

\vskip 2mm

\vskip 2mm
\noindent
{\bf Remark}
\vskip 2mm
\noindent When $j = n-1$, Theorem~\ref{WkkVal} states that for any $\alpha_1, \dots, \alpha_{n-1} \geq 0$ 
$$
\int\limits_{\partial K} H^{\frac1{n+1}}_{n-1}(x) \prod_{i=1}^{n-1} k_{i}^{\alpha_i} (x) \, d\mathcal{H}^{n-1}(x)
$$
is a valuation on the set~$\mathcal{C}^n$. In particular, when $\alpha_1 = \dots = \alpha_{n-1} = \frac s{n+1}$ for any $s>0$, then
$$
\int\limits_{\partial K} H^{\frac1{n+1}}_{n-1}(x) \left(\prod_{i=1}^{n-1} k_{i} (x) \right)^{\frac s{n+1}} \, d\mathcal{H}^{n-1}(x)
= \int\limits_{\partial K} H^{\frac{s+1}{n+1}}_{n-1}(x) \, d\mathcal{H}^{n-1}(x) = as_{1, s}(K)
$$
and Theorem~\ref{WkkVal} implies that mixed affine surface areas $as_{1,s}$ are valuations.
It seems that this fact was not known before. 

\vskip 2mm

%%%%%%%%%%%%%%%%%%%%%%%%%%%%%%%%%%%%%%%%%%%%%%%
%%%%%%%%%%%%%%%%%%%%%%%%%%%%%%%%%%%%%%%%%%%%%%%%%

\subsection{Continuity}\label{continuity}
 
For convex bodies $K$ and $L$, their Hausdorff distance is
\begin{equation}\label{Hausdorff}
d_H(K,L)=\min\{\varepsilon: K\subset L+\eps B_2^n, L\subset K+\eps B_2^n\}.
\end{equation}
It was proved by Lutwak \cite{Lutwak96} (see also \cite{Ludwig}) that for $p\geq 1$, $L_p$ affine surface area is  an upper semi continuous functional with respect to the Hausdorff metric.
In fact, it follows from Lutwak's proof that the same holds for all $0 \leq p <1$ (aside from the case $p=0$, which is just volume and hence continuous).
For $-n < p \leq 0$, the functional is lower semi continuous as was shown by Ludwig \cite{Ludwig2010}. It was also shown there that the functional defined by \eqref{pasa} is not lower semi continuous when  $-n < p \leq 0$, and that it is lower semi continuous  for $p < -n$ whereas such a statement is not true in this range for the functional defined by \eqref{def:paffine}.
\par
\noindent
As $\mathcal{V}_0^p(K) = \mathcal{W}^p_{0,\, 0} = as_p(K)$, 
it is natural to ask about the continuity properties for the $L_p$-Steiner quermassintegrals $\mathcal{V}_k^p$ and the $L_p$ Steiner coefficients 
$\mathcal{W}^p_{m,\, k}$. 
Of course, if $p=0$, then $\mathcal{V}_0^0(K) = n \, \vol_{n}(K)$, which is continuous.
\vskip 3mm
\vskip 1mm
\noindent  We first consider the cases when the assumption $K \in \mathcal{C}^n$ is not needed. This holds for $\mathcal{W}^p_{0,\, k}$ \eqref{m=0}. As these coincide with the mixed $L_p$ affine surface areas $as_{p+\frac{k}{n} (n+p), -k}$, this shows the continuity properties of the latter, which, as far as we know, had not been stated before.

\begin{theorem}  Let $k \geq 0$.
\begin{itemize}
\item[(i)] 
For $0 \leq p \leq \infty$, $\mathcal{W}^p_{0,\, k}$ is upper semi continuous on the set $\mathcal{K}^n_o$  and for $-n < p \leq 0$,
lower semi continuous on the set $\mathcal{K}^n_o$.
\item [(ii)] The $as_{p, s}(K)$ are upper semi continuous  on the set $\mathcal{K}^n_o$ when $s <0$ and $p > -s$, and lower semi continuous  on the set $\mathcal{K}^n_o$ if $s<0$ and $-n < p < -s$.
%\item [(iv)] 
\end{itemize}
\end{theorem}
\begin{proof} 
	We prove part (i). Part~(ii) follows similarly.
We note that if $K_l$ converges to $K$ in the Hausdorff metric, then $h_{K_l}$ converges to $h_K$ uniformly on $S^{n-1}$. 
Using~\ref{m=0}, the proof of (i) then follows from the proof of upper semi continuity of a generalized $L_p$ affine surface area 
given in \cite{Ludwig} where we take $f(t) = t^\frac{p}{n+p}$ that satisfies the conditions in \cite{Ludwig} (since $ 0 < \frac{p}{n+p} < 1$).
\end{proof}
\vskip 3mm
\noindent
The next corollary follows immediately from this theorem and the previous sections. A characterization of upper semi continuous rigid motion invariant valuations on convex bodies in the plane was given in \cite{Ludwig2000}.
\vskip 2mm

%\noindent
\begin{cor}
Let $p \geq 0$. Then $\mathcal{W}^p_{0,\, k} = as_{p+\frac{k}{n} (n+p), -k}$ are upper semi continuous $n \frac{n-p}{n+p} ~-~k$ homogeneous valuations on the set $\mathcal{K}^n_o$ that are invariant under rotations and reflections. 
\end{cor}
\vskip 3mm
\noindent
We will repeatedly use the following example:
\newline
For $x \in \mathbb{R}^n$, let $ \|x\|_\infty = \max_{1 \leq i \leq n} |x_i|$ and 
let $B^n_\infty=\{x \in \mathbb{R}^n: \|x\|_\infty \leq 1\}$.  
For $ l \in \mathbb{N}$, we consider convex bodies $K_l$.
We describe $K_l$  for the first 
quadrant $\mathbb{R}^n_+$, and call them $K_l^+$. The other quadrants are described accordingly. Let $x_0 = (1-\frac{1}{l}, \cdots, 1-\frac{1}{l})$ and $B^n_2\left(x_0, \frac{1}{l}\right)$ be the Euclidean ball centered at $x_0$ with radius $\frac{1}{l}$. We put 
\begin{equation}\label{Beispiel}
	K_l^+ =\left(1- \frac1l\right)B^n_\infty + \frac1l B^n_2.
\end{equation}

\vskip 2mm
\noindent
{\bf Remark}
\vskip 2mm
\noindent
Let $K_l$ be as in (\ref{Beispiel}). 
Then $K_l \rightarrow B^n_\infty$ in the Hausdorff metric and $\mathcal{W}^p_{0,\, k} (B^n_\infty)=0$ for $- \infty \leq p < -n$.
Moreover,
\begin{eqnarray*}
\mathcal{W}^p_{0,\, k} (K_l) &=&\int\limits_{\partial K_l} H_{n-1}^{\frac p{n+p}}  \langle x, N(x) \rangle^{-k - n\frac{p-1}{n+p}} d\mathcal{H}^{n-1}(x) \\
&\geq& n^{\frac{1}{2}\left(-k - n\frac{p-1}{n+p}\right) } \,  l^{-n\frac{n-p}{n+p} }  \vol_{n-1}  \,  \,  (\partial B^n_2) \rightarrow \infty,
\end{eqnarray*} 
as $l \rightarrow \infty$. This shows that  $\mathcal{W}^p_{0,\, k}$ is not upper semi continuous for $- \infty \leq p < -n$. It is also not lower semi continuous 
for that $p$-range as is easily seen by taking a sequence of polytopes that converges to the Euclidean unit ball.

\vskip 3mm
\noindent
\noindent While there are  continuity properties of the $L_p$ Steiner coefficients for $m=0$, we cannot expect that any of the other $L_p$ Steiner coefficients or $L_p$-Steiner quermassintegrals have  continuity properties. 
\medskip

\noindent We note that $\mathcal{C}^n$ is not closed under the topology generated by the Hausdorff metric. Thus, when addressing continuity issues,  we consider  convergence in the Hausdorff topology involving bodies for which the $L_p$ Steiner coefficients are well-defined  but such that these bodies are not necessarily in~$\mathcal{C}^n$.

\vskip 2mm

\begin{prop} 
\item Let $k \geq 1$. Then 
$\mathcal{V}^1_{k}= \mathcal{W}_{k,\, k}$ are  neither  lower semi continuous nor upper semi continuous. % in general.
\end{prop}
\begin{proof} 
It is clear that $\mathcal{W}_{k,\, k}$ is not continuous with respect to the Hausdorff metric: take a sequence of polytopes that approximates the Euclidean ball.

Let $K_l$ be as in (\ref{Beispiel}). 
Then $K_l \rightarrow B^n_\infty$ in the Hausdorff metric  and  $\mathcal{W}_{k,\, k}(B^n_\infty)=0$. Moreover, even though $K_l$ is not in $\mathcal{C}_p$, the coefficients $\mathcal{W}_{k,\, k}(K_l)$ are well-defined. Indeed, using (\ref{Kugel}), 
\begin{eqnarray}\label{Kl}
\mathcal{W}_{k,\, k}(K_l) &=&\int\limits_{\partial K_l} H_{n-1}^{\frac 1{n+1}}  \sum_{\substack{
				i_1, \dots, i_{n-1} \geq 0 \\
				i_1 + 2i_2 + \dots + (n-  1)i_{n-1}=k}}
		c(n, 1, i(k))  \prod\limits_{j = 1}^{n - 1}  {n - 1\choose j}^{i_j} H_{j}^{i_j}\left(x \right) \, d\mathcal{H}^{n-1}(x) \nonumber  \\
		&=& \int\limits_{\partial B^n_2\left(0, \frac{1}{l}\right)} H_{n-1}^{\frac 1{n+1}}  \sum_{\substack{
				i_1, \dots, i_{n-1} \geq 0 \\
				i_1 + 2i_2 + \dots + (n-1)i_{n-1}=k}}
		c(n, 1, i(k))  \prod\limits_{j = 1}^{n - 1}  {n - 1\choose j}^{i_j} H_{j}^{i_j}\left(x \right) \, d\mathcal{H}^{n-1}(x) \nonumber  \\
&=& \mathcal{W}_{k,\, k}\left( B^n_2\left(0, \frac{1}{l}\right)\right)\nonumber \\
&=& l^{k-n\frac{n-1}{n+1} }  \vol_{n-1}  \,  \,  (\partial B^n_2) \,  C(n,1,k).
\end{eqnarray} 
Let first $k \geq n$. If $k-n+1$ is even, then we have by Proposition  \ref{Ckk} that $C(n,1,k) >0$.  Therefore, as $k-n \frac{n-1}{n+1} >0$,
$$
\mathcal{W}_{k,\, k}(K_l)= l^{k-n\frac{n-1}{n+1} }  \vol_{n-1}  \,  \,  (\partial B^n_2) \,  C(n,1,k) \to \infty, 
$$
as $l \to \infty$, but $\mathcal{W}_{k,\, k}(B^n_\infty)=0$. Thus $\mathcal{W}_{k,\, k}$ is not upper semi continuous in this case.
$\mathcal{W}_{k,\, k}$ is also not lower semi continuous. We take a sequence $P_l$ of polytopes that converges to $B^n_2$
in the Hausdorff metric.
Then $\mathcal{W}_{k,\, k}(P_l)=0$ for all $l$ but $\mathcal{W}_{k,\, k}(B^n_2)>0$.
\par
\noindent
Let now $k \geq n$ be such that  $k-n+1$ is odd. Then $C(n,1,k) < 0$ by Proposition  \ref{Ckk} and thus by 
 (\ref{unitball}), $\mathcal{W}_{k,\, k}(B^n_2)= \vol_{n-1}(\partial B^n_2) \, C(n,1,k) <0$. 
As $\mathcal{W}_{k,\, k}(P_l)=0$ for all $l$, where $P_l$ is a  sequence of polytopes  that converges to $B^n_2$ in the Hausdorff metric, 
this shows that $\mathcal{W}_{k,\, k}$ is not lower semi continuous in this case.
$\mathcal{W}_{k,\, k}$ is also not upper semi continuous. Indeed, $K_l \to B^n_\infty$ as $l \to \infty$ in the Hausdorff metric and
$\mathcal{W}_{k,\, k}(B^n_\infty)=0$. But 
$$
\mathcal{W}_{k,\, k}(K_l)= l^{k-n\frac{n-1}{n+1} }  \vol_{n-1}  \,  \,  (\partial B^n_2) \,  C(n,1,k) \to  -\infty.
$$
This settles negatively the semi continuity of $\mathcal{W}_{k,\, k}$  for all $n \geq 2$ and all $k \geq n$.
\vskip 2mm
\noindent
Let now $k \leq n-1$. By (\ref{unitball}), we have that $\mathcal{W}_{k,\, k}(B^n_2)= \vol_{n-1}(\partial B^n_2) \, C(n,1,k)>0$. 
As $\mathcal{W}_{k,\, k}(P) =0$ for every polytope $P$, and as by 
Proposition  \ref{Ckk},   $ C(n,1,k) >0$ for all $n \geq 2$,  this shows that $\mathcal{W}_{k,\, k}$ is not lower semi continuous in this  $k$-range.
Moreover, by Proposition \ref{Ckk},  for every $n\geq 2$,   $C(n,1,n-1) >0$.  
Therefore,  for all $n \geq 2$
$$
\mathcal{W}_{n-1,\, n-1}(K_l)  = l^{\frac{n-1}{n+1} } \, \vol_{n-1}  \,  \,  (\partial B^n_2) \,  C(n,1,n-1) \rightarrow  \infty,
$$ 
as $l \rightarrow \infty$ and  thus $\mathcal{W}_{n-1,\, n-1}$ is also not  upper  semi continuous, which settles 
negatively the semi continuity of $\mathcal{W}_{k,\, k}$  for all $n \geq 2$ and all $k \geq n-1$.
In particular, the case $n=2$ is settled for all $k$.
\vskip 2mm
\noindent
To settle upper semi continuity for all dimensions and all $k< n-1$  requires further examples. 
\medskip

\noindent For instance, for $n=3$,  only upper semi continuity for $k=1$ is not yet settled. To resolve this, we consider
$$
\mathcal{W}_{1,\, 1}(K)  = \frac{3}{2} \int_{\partial K} H_2(x) ^\frac{1}{4} \ H_1(x) \ d\mathcal{H}^{2}(x).
$$
We place the body $K_l$ in the $xz$-plane and we let $Z_l \subset \mathbb{R}^3$ be the  body of revolution  obtained by rotating  $K_l $ about the $z$-axis.
Let $x \in \partial Z_l$. Then the maximal principle curvature $\kappa_1(x) = l$ and  the minimal principle curvature $\kappa_2(x) \geq 1$. Therefore, 
$$
 H_2(x) ^\frac{1}{4} \geq l^\frac{1}{4} \hskip 3mm \text{and} \hskip 3mm H_1(x) \geq \frac{1}{2} \left( l +1\right)
 $$
and consequently
$$
\mathcal{W}_{1,\, 1}(Z_l)  \geq  \frac{3}{4} l^\frac{1}{4}   \left( l +1\right)  \int_{\partial Z_l}  d\mathcal{H}^{2}(x) \geq  3\pi l^\frac{1}{4}   \left( l +1\right) \frac{1-\frac{1}{l}}{l}.
$$
Thus $\mathcal{W}_{1,\, 1}(Z_l) \to \infty$, as $l \to \infty$, but $Z_l$ converges in the Hausdorff metric to the cylinder $Z$ of height $2$ and a $2$-dimensional Euclidean unit ball as base and $\mathcal{W}_{1,\, 1}(Z)=0$. Thus $\mathcal{W}_{1,\, 1}$ is not upper semi continuous and this settles $n=3$ for all $k$.
\par
\noindent
Modifications of these examples  show that $\mathcal{W}_{k,\, k}$ is not upper semi continuous for all $n \geq 4$ and all $1 \leq k< n-1$.
\par

\end{proof}

\vskip 3mm
\noindent
Now we treat general parameters $p$ such that $\frac{p}{n+p} \geq 0$.   $\mathcal{V}_{0}^0 (K)= n\,  \vol_{n}(K)$  and  $\mathcal{V}_{1}^0(K) =c(n) \vol_{n-1}(\partial K)$, where $c(n)$ is a constant depending only on $n$.
Thus, $\mathcal{V}_{0}^0$ and $\mathcal{V}_{1}^0$ are continuous. 
Moreover, $\mathcal{V}_{0}^p (K)= as_p(K)$ which, as noted above, is upper semi continuous when $p\geq 0$
and  lower semi continuous for  $-n < p \leq 0$.
\vskip 2mm
\noindent
In general however, we do not have any continuity property.
\par
\noindent
\begin{prop} 
Let $- \infty < p < -n$ or $0 \leq p < \infty$, and let $k \geq 1$. Let $K$ be such that $\mathcal{V}_{k}^p$ are well-defined. Then 
$\mathcal{V}_{k}^p$ are  in general neither  lower semi continuous nor upper semi continuous. 
\end{prop}
\begin{proof} 
If $p=0$ and $k \geq 2$, then by (\ref{V0-poly}),  $\mathcal{V}_{k}^0(P)=0$ for every polytope $P$ and by Corollary \ref{Kugel0+n}, 
$\mathcal{V}^0_{k}(B^n_2) = {n \choose k} \vol_{n-1}(\partial B^n_2)$ for $2 \leq k \leq n$. Thus, taking a sequence $P_l$ of polytopes that converges to $B^n_2$, this shows that $\mathcal{V}^0_{k}$ is not  lower semi continuous for  $2 \leq k\leq n$.
\newline
If $p\neq 0$, then by \eqref{Vp-poly},  $\mathcal{V}_{k}^p(P)=0$ for every polytope $P$ and if $p <-n$ or $p >n$ then  by (\ref{Kugel0}) and Proposition \ref{Cnpk}, 
\begin{equation*}
\mathcal{V}^p_{k}(B^n_2) = C(n,p,k) \, \vol_{n-1}(\partial B^n_2) \left\{
\begin{aligned} & < 0  \quad &\text{ if } \,  k  \, &\text{ is odd } \\ 
&   >0  \quad &\text{ if } \,  k  \, &\text{ is even }.
\end{aligned}\right.
\end{equation*}
Taking a sequence $P_l$ of polytopes that converges to $B^n_2$, shows that for $-\infty \leq p <-n$ or $n <p \leq \infty$, $\mathcal{V}^p_{k}$ is not upper semi continuous when $k$ is odd
and not lower semi continuous when $k$ is even. Similarly, absence of upper  resp. lower semi continuity can be determined for the other $p$-ranges, using Proposition \ref{Cnpk}.
\par
\noindent
For the semi continuity issues not yet settled, we will now use the convex bodies  $K_l$, $l \in \mathbb{N}$  given by (\ref{Beispiel}). To determine $\mathcal{V}^p_{k}(K_l)$,  it is enough  in this case to 
consider $B_l=\partial K_l \cap \mathbb{R}^n_+ \cap \partial B^n_2(x_0, \frac{1}{l})$, where $x_0= \left(1-\frac{1}{l}\right) ( 1, 1, \cdots, 1)$. 
Note that for $x= x_0+ \frac{1}{l} \xi \in B_l$, $\xi \in S^{n-1}_+$, we have
\begin{equation}\label{support}
\langle x, N(x)\rangle = h_{K_l}(\xi)= \left(1-\frac{1}{l}\right) \sum_{i=1}^n \xi_i + \frac{1}{l}.
\end{equation}
Then
\begin{eqnarray*}
\mathcal{V}_{k}^p(K_l)
&=& 2^n  \sum\limits_{m = 0}^{k} \binom{\frac{n(1-p)}{n+p}}{k-m} \int\limits_{S^{n-1}_+} 
h_{K_l} ^{m - k + \frac{n(1-p)}{n+p}} \, \frac{l^{(n-1) \frac p{n+p}+m}}{l^{n-1}} \nonumber \\
&& \sum_{\substack{
			i_1, \dots, i_{n-1} \geq 0 \nonumber \\
			i_1 + 2i_2 + \dots + (n-  1)i_{n-1}=m}}
	 c(n, p, i(m)) \prod\limits_{j = 1}^{n - 1}  {n - 1\choose j}^{i_j} \, d\mathcal{H}^{n-1} \nonumber \\
&=&	 2^n  \sum\limits_{m = 0}^{k}   l^{-\frac{n(n-1)}{n+p}+m} \binom{\frac{n(1-p)}{n+p}}{k-m} \sum_{\substack{
			i_1, \dots, i_{n-1} \geq 0 \nonumber \\
			i_1 + 2i_2 + \dots + (n-  1)i_{n-1}=m}}
	 c(n, p, i(m)) \prod\limits_{j = 1}^{n - 1}  {n - 1\choose j}^{i_j} \,    \\
	 && \int\limits_{S^{n-1}_+} \left(\left(1-\frac{1}{l}\right) 
\sum_{i=1}^n \xi_i + \frac{1}{l} \right) ^{m - k + \frac{n(1-p)}{n+p}}  \, d\mathcal{H}^{n-1},	
	\end{eqnarray*}
where we used \eqref{LpSteiner-Ein} and \eqref{support} in the first equality. This shows that $\mathcal{V}_{k}^p(K_l)$ are finite.
Therefore,
\begin{eqnarray*}
\lim_{ l \to \infty} \mathcal{V}_{k}^p(K_l)&=&  2^n  \lim_{ l \to \infty} \sum\limits_{m = 0}^{k} l^{-\frac{n(n-1)}{n+p}+m} \binom{\frac{n(1-p)}{n+p}}{k-m} \sum_{\substack{
			i_1, \dots, i_{n-1} \geq 0 \nonumber \\
			i_1 + 2i_2 + \dots + (n-  1)i_{n-1}=m}}
	 c(n, p, i(m))    \prod\limits_{j = 1}^{n - 1}  {n - 1\choose j}^{i_j} \nonumber \\
	 && \lim_{ l \to \infty} \int\limits_{S^{n-1}_+} \left(\left(1-\frac{1}{l}\right) 
\sum_{i=1}^n \xi_i + \frac{1}{l} \right) ^{m - k + \frac{n(1-p)}{n+p}}  \, d\mathcal{H}^{n-1}.	
\end{eqnarray*}
As for $\xi \in S^{n-1}_+$,  $1 \leq \sum_{i=1}^n \xi_i \leq \sqrt{n}$, the functions under the integral are uniformly in $l$ bounded by an integrable function
and by Lebegue's Dominated Convergence theorem we can interchange integration and limit. Therefore
\begin{eqnarray*}
&& \lim_{ l \to \infty} \int\limits_{S^{n-1}_+} \left(\left(1-\frac{1}{l}\right) 
\sum_{i=1}^n \xi_i + \frac{1}{l} \right) ^{m - k + \frac{n(1-p)}{n+p}}  \, d\mathcal{H}^{n-1} = \\
&& \int\limits_{S^{n-1}_+}  \lim_{ l \to \infty}  \left(\left(1-\frac{1}{l}\right) 
\sum_{i=1}^n \xi_i + \frac{1}{l} \right) ^{m - k + \frac{n(1-p)}{n+p}}  \, d\mathcal{H}^{n-1} =  \int\limits_{S^{n-1}_+} \left( \sum_{i=1}^n \xi_i  \right) ^{m - k + \frac{n(1-p)}{n+p}}  \, d\mathcal{H}^{n-1}.
\end{eqnarray*}
Hence
\begin{eqnarray}\label{Beispiel-p-Fall}
\lim_{ l \to \infty} \mathcal{V}_{k}^p(K_l) &= & 2^n  \lim_{ l \to \infty} \sum\limits_{m = 0}^{k}   l^{-\frac{n(n-1)}{n+p}+m} \binom{\frac{n(1-p)}{n+p}}{k-m} \sum_{\substack{
			i_1, \dots, i_{n-1} \geq 0 \nonumber \\
			i_1 + 2i_2 + \dots + (n-  1)i_{n-1}=m}}
	 c(n, p, i(m))  \,   \prod\limits_{j = 1}^{n - 1}  {n - 1\choose j}^{i_j}   \nonumber \\
&&	  \int\limits_{S^{n-1}_+} \left( \sum_{i=1}^n \xi_i  \right) ^{m - k + \frac{n(1-p)}{n+p}}  \, d\mathcal{H}^{n-1}.
\end{eqnarray}
\par
\noindent
We look now in more detail at the  case $p <-n$.
In this case, the exponent  $-\frac{n(n-1)}{n+p}+m$ of $l$  in (\ref{Beispiel-p-Fall})  is strictly positive.
We observe that for $\alpha \in \mathbb{N} \cup \{0\}$
\begin{equation*}
\binom{\frac{n(1-p)}{n+p}}{\alpha} \left\{
\begin{aligned} 
& = 0  \quad &\text{ if } \,  \alpha \, &= 0 \\ 
& < 0  \quad &\text{ if } \,  \alpha  \, &\text{ is odd } \\ 
&   >0  \quad &\text{ if } \,  \alpha  \, &\text{ is even}.
\end{aligned}\right.
\end{equation*}
The analysis of the expression 
$$
F_m(p) =\sum\limits_{\substack{
			i_1, \dots, i_{n-1} \geq 0 \nonumber \\
			i_1 + 2i_2 + \dots + (n-  1)i_{n-1}=m}}
	 c(n, p, i(m))  \,   \prod\limits_{j = 1}^{n - 1}  {n - 1\choose j}^{i_j} $$
is more involved. We have that
\begin{equation*}\label{m=1}
F_1(p) =  (n-1)\, \frac{n}{n+p}	 
\end{equation*}
 which is increasing and $F_1(p)  \leq 0$ on $[-\infty, -n)$.
\begin{equation*}\label{m=2}
F_2(p) = \frac{(n-1)}{2!}\, \frac{n}{n+p} \left((n-1) \frac{n}{n+p}-1\right),
\end{equation*}
which is decreasing and $F_2(p)  \geq 0$ on $[-\infty, -n)$. 
\begin{equation*}\label{m=3}
F_3(p) =  \frac{(n-1)}{3!}\, \frac{n}{n+p} \left((n-1)^2 \left(\frac{n}{n+p}\right)^2 -3 (n-1) \frac{n}{n+p}+2 \right),
\end{equation*}
which is increasing  and  $F_3(p) \leq 0$ on $[-\infty, -n)$.

\noindent
From the analysis of the signs of $\binom{\frac{n(1-p)}{n+p}}{\alpha}$ and the $F_m$ we  conclude from (\ref{Beispiel-p-Fall}) that
$$
\lim_{ l \to \infty} \mathcal{V}_{1}^p(K_l)= - \infty
$$
and thus $\mathcal{V}_{1}^p$ is not lower semi continuous for $p \in [-\infty, -n)$,
$$
\lim_{ l \to \infty} \mathcal{V}_{2}^p(K_l)=  \infty
$$
and thus $\mathcal{V}_{2}^p$ is not upper semi continuous for $p \in [-\infty, -n)$,
$$
\lim_{ l \to \infty} \mathcal{V}_{3}^p(K_l)= - \infty
$$
and thus $\mathcal{V}_{3}^p$ is not lower semi continuous for $p \in [-\infty, -n)$. Starting from $m=4$, the behavior of $F_m$ is not monotone anymore on $[-\infty, 0)$.

\noindent
Similarly, for  $n <p < n(n-2)$, 
$$
\lim_{ l \to \infty} \mathcal{V}_{\frac{n(n-1)}{n+p}}^p(K_l) = F_{\frac{n(n-1)}{n+p}} (p) \int\limits_{S^{n-1}} \left( \sum_{i=1}^n \xi_i  \right) ^{\frac{n(1-p)}{n+p}}  \, d\mathcal{H}^{n-1} 
$$ is a positive or negative number,
 which disproves upper or lower semi continuity, depending on the sign of $F_{\frac{n(n-1)}{n+p}}(p)$. 

\end{proof}

\vskip 5mm

\subsection*{Acknowledgement}
The second author wants to thank the Hausdorff Research Institute for Mathematics. Part of the work on the paper was carried out during her stay there. The authors also wish to thank the Institute for Computational and Experimental Research in Mathematics (ICERM) for the hospitality. The manuscript was completed during their participation at the program ``Harmonic Analysis and Convexity".

	\begin{center}
		\begin{tabular}{l l}
			\multirow{5}{17em}{\addressT} & \qquad \qquad  \multirow{5}{17em}{\addressW}
		\end{tabular}
	\end{center}	
	
\end{document}